\documentclass{amsart}
\pdfoutput=1

\usepackage{amssymb}
\usepackage{microtype}
\usepackage{tikz}
\usepackage[pdftex,colorlinks,citecolor=black,linkcolor=black,urlcolor=black,bookmarks=false]{hyperref}
\def\MR#1{\href{http://www.ams.org/mathscinet-getitem?mr=#1}{MR#1}}
\def\arXiv#1{arXiv:\href{http://arXiv.org/abs/#1}{#1}}
\usepackage{doi}

\numberwithin{equation}{section}
\numberwithin{figure}{section}

\newtheorem{theorem}{Theorem}[section]
\newtheorem{proposition}[theorem]{Proposition}
\newtheorem{lemma}[theorem]{Lemma}
\newtheorem{corollary}[theorem]{Corollary}

\theoremstyle{definition}
\newtheorem{definition}[theorem]{Definition}

\theoremstyle{remark}
\newtheorem{remark}[theorem]{Remark}

\newcommand{\abs}[1]{\left\lvert#1\right\rvert}
\newcommand{\sabs}[1]{\lvert#1\rvert}
\newcommand{\norm}[1]{\left\lVert#1\right\rVert}
\newcommand{\snorm}[1]{\lVert#1\rVert}
\newcommand{\ang}[1]{\left\langle #1 \right\rangle}

\newcommand{\ceil}[1]{\left\lceil #1 \right\rceil}
\newcommand{\paren}[1]{\left( #1 \right)}

\newcommand{\set}[1]{\left\{ #1 \right\}}

\newcommand{\wt}{\widetilde}

\DeclareMathOperator{\sign}{sign}
\newcommand{\symmdiff}{\mathbin{\triangle}}

\newcommand{\x}{\times}
\newcommand{\e}{\varepsilon}

\newcommand{\EE}{\mathbb{E}}
\newcommand{\RR}{\mathbb{R}}
\newcommand{\PP}{\mathbb{P}}

\newcommand{\cB}{\mathcal{B}}
\newcommand{\cP}{\mathcal{P}}
\newcommand{\cQ}{\mathcal{Q}}
\newcommand{\cR}{\mathcal{R}}
\newcommand{\bG}{\mathbf{G}}
\newcommand{\bH}{\mathbf{H}}

\title[An $L^p$ theory of sparse graph convergence I]{An $L^p$ theory
  of sparse graph convergence I: limits, sparse random
  graph models,\\ and power law distributions}

\author{Christian Borgs}
\address{Microsoft Research\\
One Memorial Drive\\
Cambridge, MA 02142} \email{borgs@microsoft.com}

\author{Jennifer T.\ Chayes}
\address{Microsoft Research\\
One Memorial Drive\\
Cambridge, MA 02142} \email{jchayes@microsoft.com}

\author{Henry Cohn}
\address{Microsoft Research\\
One Memorial Drive\\
Cambridge, MA 02142} \email{cohn@microsoft.com}

\author{Yufei Zhao}
\address{Department of Mathematics\\
Massachusetts Institute of Technology\\
Cambridge, MA 02139} \email{yufeiz@mit.edu}
\thanks{Zhao was supported by a
Microsoft Research PhD Fellowship and internships at Microsoft Research New England.}

\begin{document}

\begin{abstract}
We introduce and develop a theory of limits for sequences of sparse graphs
based on $L^p$ graphons, which generalizes both the existing $L^\infty$
theory of dense graph limits and its extension by Bollob\'as and Riordan to
sparse graphs without dense spots.  In doing so, we replace the no dense
spots hypothesis with weaker assumptions, which allow us to analyze graphs
with power law degree distributions.  This gives the first broadly applicable
limit theory for sparse graphs with unbounded average degrees.  In this
paper, we lay the foundations of the $L^p$ theory of graphons, characterize
convergence, and develop corresponding random graph models, while we prove
the equivalence of several alternative metrics in a companion paper.
\end{abstract}

\setcounter{tocdepth}{1}
\maketitle

\tableofcontents

\section{Introduction}

Understanding large networks is a fundamental problem in modern graph theory.
What does it mean for two large graphs to be similar to each other, when they
may differ in obvious ways such as their numbers of vertices?  There are many
types of networks (biological, economic, mathematical, physical,
social, technological, etc.), whose details vary widely, but similar structural and growth
phenomena occur in all these domains.  In each case, it is natural to
consider a sequence of graphs with size tending to infinity and ask whether
these graphs converge to any meaningful sort of limit.

For dense graphs, the theory of graphons provides a comprehensive and
flexible answer to this question (see, for example,
\cite{BCLSV1,BCLSV2,Lov,LS06,LS07}).  Graphons characterize the limiting
behavior of dense graph sequences, under several equivalent metrics that
arise naturally in areas ranging from statistical physics to combinatorial
optimization.  Because dense graphs have been the focus of much of the graph
theory developed in the last half century, graphons and
related structural results about dense graphs play a foundational role in
graph theory. However, many large networks of interest in other
fields are sparse, and in the dense theory all sparse graph sequences
converge to the zero graphon. This greatly limits the applicability of
graphons to real-world networks. For example, in statistical physics dense
graph sequences correspond to mean-field models, which are conceptually
important as limiting cases but rarely applicable in real-world systems.

At the other extreme, there is a theory of graph limits for very sparse
graphs, namely those with bounded degree or at least bounded average degree
\cite{AL,AlSt,BS,Lyons} (see also \cite{NOdM1,NOdM2,NOdM3,NOdM4} for a
broader framework based on first-order logic). Although this theory covers
some important physical cases, such as crystals, it also does not apply to
most networks of current interest. And although it is mathematically
completely different in spirit from the theory of dense graph limits, it is
also limited in scope. It covers the case of $n$-vertex graphs with $O(n)$
edges, while dense graph limits are nonzero only when there are $\Omega(n^2)$
edges.

Bollob\'as and Riordan \cite{BR} took an important step towards bridging the
gap between these theories.  They adapted the theory of graphons to sparse
graphs by renormalizing to fix the effective edge density, which captures the
intuition that two graphs with different densities may nevertheless be
structurally similar. Under a boundedness assumption (Assumption~4.1 in
\cite{BR}), which says that there are no especially dense spots within the
graph, they showed that graphons remain the appropriate limiting objects. In
other words, sparse graphs without dense spots converge to graphons after
rescaling.  Thus, these sparse graph sequences are characterized by their
asymptotic densities and their limiting graphons.

The Bollob\'as-Riordan theory extends the scope of graphons to sparse graphs,
but the boundedness assumption is nevertheless highly restrictive.  In loose terms, it
means the edge densities in different parts of the graph are all on roughly
the same scale. By contrast, many of the most exciting network models have
statistics governed by power laws \cite{CSN,Mitz}.  Such models generally
contain dense spots, and we therefore must broaden the theory of graphons to handle
them.

One setting in which these difficulties arise in practice is statistical
estimation of network structure. Each graphon has a corresponding random
graph model converging to it, and it is natural to try to fit these models
to an observed network and thus estimate the underlying graphon (see, for
example, \cite{BC}).  Using the Bollob\'as-Riordan theory, Wolfe and Olhede
\cite{WO} developed an estimator and proved its consistency under certain
regularity conditions.  Their theorems provide valuable statistical tools,
but the use of the Bollob\'as-Riordan theory limits the applicability of
their approach to graphs without dense spots and thus excludes many
important cases.

In this paper, we develop an $L^p$ theory of graphons for all $p>1$, in
contrast with the $L^\infty$ theory studied in previous papers.\footnote{The paper
\cite{KKLS} and the
online notes to Section~17.2 of \cite{Lov} go a little beyond $L^\infty$
graphons to study graphons in $\bigcap_{1\le p < \infty} L^p$.} The $L^p$
theory provides for the first time the flexibility to account for power laws,
and we believe it is the right convergence theory for sparse graphs (outside
of the bounded average degree regime). It generalizes dense graph limits and
the Bollob\'as-Riordan theory, which together are the special case
$p=\infty$, and it extends all the way to the natural barrier of $p=1$.

It is also worth noting that, in the process of developing an $L^p$ theory
of graphons, we give a new $L^p$ version of the Szemer\'edi regularity lemma for
all $p>1$ in its so-called weak (integral) form, which also naturally suggests
the correct formulation for stronger forms.  Long predating the theory of graph
limits and graphons, it was recognized that the regularity lemma is a cornerstone
of modern graph theory and indeed other aspects of discrete mathematics, so attempts
were made to  extend it to non-dense graphs.  Our  $L^p$ version of the weak
Szemer\'edi regularity lemma  generalizes and  extends previous work,
as discussed below.

We will give precise definitions and theorem statements in
\textsection\ref{sec:defres}, but first we sketch some examples motivating
our theory.

We begin with dense graphs and $L^\infty$ graphons.  The most basic random
graph model is the Erd\H os-R\'enyi model $G_{n,p}$, with $n$ vertices and
edges chosen independently with probability $p$ between each pair of
vertices.  One natural generalization replaces $p$ with a symmetric $k \times
k$ matrix; then there are $k$ blocks of $n/k$ vertices each, with edge
density $p_{i,j}$ between the $i$-th and $j$-th blocks.  As $k \to \infty$,
the matrix becomes a symmetric, measurable function $W \colon [0,1]^2 \to
[0,1]$ in the continuum limit.  Such a function $W$ is an $L^\infty$ graphon.
All large graphs can be approximated by $k \times k$ block models with $k$
large via Szemer\'edi regularity, from which it follows that limits of dense
graph sequences are $L^\infty$ graphons.

For sparse graphs the edge densities will converge to zero, but we would like
a more informative answer than just $W=0$.  To determine the asymptotics, we
rescale the density matrix $p$ by a function of $n$ so that it no longer
tends to zero. In the Bollob\'as-Riordan theory, the boundedness assumption
ensures that the densities are of comparable size (when smoothed out by local
averaging) and hence remain bounded after rescaling. They then converge to an
$L^\infty$ graphon, and the known results on $L^\infty$ graphons apply modulo
rescaling.

For an example that cannot be handled using $L^\infty$ graphons, consider the
following configuration model.  There are $n$ vertices numbered $1$ through
$n$, with probability $\min(1,n^\beta (ij)^{-\alpha})$ of an edge between $i$
and $j$, where $0 < \alpha < 1$ and $0 \le \beta < 2\alpha$.  In other words,
the probabilities behave like $(ij)^{-\alpha}$, but boosted by a factor of
$n^\beta$ in case they become too small.\footnote{The inequalities $\alpha<1$
and $\beta<2\alpha$ each have a natural interpretation: the first avoids
having almost all the edges between a sublinear number of vertices, while the
second ensures that the cut-off from taking the minimum with $1$ affects only
a negligible fraction of the edges.}  This model is one of the simplest ways
to get a power law degree distribution, because the expected degree of vertex
$i$ scales according to an inverse power law in $i$ with exponent $\alpha$.
The expected number of edges is on the order of $n^{\beta-2\alpha+2}$, which
is superlinear when $\beta > 2\alpha-1$. However, rescaling by the edge
density $n^{\beta-2\alpha}$ does not yield an $L^\infty$ graphon. Instead, we
get $W(x,y) = (xy)^{-\alpha}$, which is unbounded.

Unbounded graphons are of course far more expressive than bounded graphons,
because they can handle an unbounded range of densities simultaneously.  This
issue does not arise for dense graphs: without rescaling, all densities are
automatically bounded by $1$. However, unboundedness is ubiquitous for
sequences of sparse graphs.

To deal with unbounded graphons, we must reexamine the foundations of the
theory of graphons.  To have a notion of density at all, a graphon must at
least be in $L^1([0,1]^2)$. Neglecting for the moment the limiting case of
$L^1$ graphons, we show that $L^p$ graphons are well behaved when $p>1$.  In
the example above, the $p>1$ case covers the full range $0 < \alpha < 1$,
and we think of it as the primary case, while $p=1$ is slightly degenerate
and requires additional uniformity hypotheses (see
Appendix~\ref{sec:uniform}).

Each graphon $W$ can be viewed as the archetype for a whole class of graphs,
namely those that approximate it.  It is natural to call these graphs
$W$-quasirandom, because they behave as if they were randomly generated using
$W$. From this perspective, the $L^p$ theory of graphons completes the
$L^\infty$ theory: it adds the missing graphons that describe sparse graphs
but not dense graphs.

The remainder of this paper is devoted to three primary tasks:
\begin{enumerate}
\item We lay the foundations of the $L^p$ theory of graphons.

\item \label{task2} We characterize the sparse graph sequences that
    converge to $L^p$ graphons via the concept of $L^p$ upper regularity,
    and we establish the theory of convergence under the cut metric.

\item \label{task3} For each $L^1$ graphon $W$, we develop sparse
    $W$-random graph models and show that they converge to $W$.
\end{enumerate}
Our main theorems are Theorems~\ref{thm:G-limit} and~\ref{thm:rand-converge},
which deal with tasks~\ref{task2} and~\ref{task3}, respectively.
Theorem~\ref{thm:G-limit} says that every $L^p$ upper regular sequence of
graphs with $p>1$ has a subsequence that converges to an $L^p$ graphon, and
Theorem~\ref{thm:rand-converge} says that sparse $W$-random graphs converge
to $W$ with probability $1$.  We also prove a number of other results, which
we state in Section~\ref{sec:defres}.  One topic we do not address here is
``right convergence'' (notions of convergence based on quotients or
statistical physics models). We analyze right convergence in detail in the
companion paper \cite{BCCZ}.

\section{Definitions and results} \label{sec:defres}

\subsection{Notation}

We consider weighted graphs, which include as a special case simple
unweighted graphs. We denote the vertex set and edge set of a graph $G$
by $V(G)$ and $E(G)$, respectively.

In a weighted graph $G$, every vertex $i \in V$ is given a weight $\alpha_i
= \alpha_i(G) > 0$, and every edge $ij \in E(G)$ (allowing loops with $i=j$)
is given a weight $\beta_{ij} = \beta_{ij}(G) \in \RR$. We set $\beta_{ij} =
0$ whenever $ij \notin E(G)$. For each subset $U \subseteq V$, we write
\[
\alpha_U := \sum_{i \in U} \alpha_i \qquad \text{and} \qquad \alpha_G := \alpha_{V(G)}.
\]
We say a
sequence $(G_n)_{n \ge 0}$ of weighted graphs has \emph{no dominant nodes} if
\[
\lim_{n \to \infty} \frac{\max_{i \in V(G_n)} \alpha_i(G_n)}{\alpha_{G_n}} = 0.
\]
A simple (unweighted) graph is one in which $\alpha_i = 1$ for all $i \in V$,
$\beta_{ij} = 1$ whenever $ij \in E$, and $\beta_{ij} = 0$ whenever $ij
\notin E$. A simple graph contains no loops or multiple edges.

For $c \in \RR$, we write $cG$ for the weighted graph obtained from $G$ by
multiplying all edge weights by $c$, while the vertex weights remain
unchanged.

For $1 \leq p \leq \infty$ we define the $L^p$ norms
\[
\norm{G}_p := \paren{\sum_{i,j \in V(G)}
\frac{\alpha_i\alpha_j}{\alpha_G^2} \sabs{\beta_{ij}}^p}^{1/p} \qquad
\text{when } 1 \leq p < \infty
\]
and
\[
\norm{G}_\infty := \max_{i,j\in V(G)} \sabs{\beta_{ij}}.
\]
The quantity $\norm{G}_1$ can be viewed as the edge density when $G$ is a
simple graph. When considering sparse graphs, we usually normalize the edge
weights by considering the weighted graph $G / \norm{G}_1$, in order to
compare graphs with different edge densities.  (Of course this assumes
$\snorm{G}_1 \ne 0$, but that rules out only graphs with no edges, and we
will often let this restriction pass without comment.)

In the previous works \cite{BCLSV1,BCLSV2} on convergence of dense graph
sequences, only graphs with uniformly bounded $\snorm{G}_\infty$ were
considered. In this paper, we relax this assumption. As we will see, this
relaxation is useful even for sparse simple graphs due to the normalization
$G / \norm{G}_1$.

Given that we are relaxing the uniform bound on $\norm{G}_\infty$, one
might think, given the title of this paper, that we impose a uniform
bound on $\norm{G}_p$. This is \emph{not} what we do. A bound on
$\norm{G}_p$ is too restrictive: for a simple graph $G$, an upper
bound on $\norm{G / \norm{G}_1}_p = \norm{G}_1^{\frac1p - 1}$ corresponds to
a lower bound on $\norm{G}_1$, which forces $G$ to be dense. Instead,
we impose an $L^p$ bound on edge densities with respect to vertex set
partitions. This is explained next.

\subsection{$L^p$ upper regular graphs} \label{sec:defres-up-reg-graph}

For any $S, T \subseteq V(G)$, define the edge density (or average
edge weight, for weighted graphs) between $S$ and
$T$ by
\[
\rho_G(S,T) := \sum_{s \in S, t \in T}
\frac{\alpha_s\alpha_t}{\alpha_S \alpha_T} \beta_{st}.
\]

We introduce the following hypothesis. Roughly speaking, it says that for
every partition of the vertices of $G$ in which no part is too small, the
weighted graph derived from averaging the edge weights with respect to the
partition is bounded in $L^p$ norm (after normalizing by the overall edge
density of the graph).

\begin{definition} \label{def:Lp-upper-reg}
A weighted graph $G$ (with vertex weights $\alpha_i$ and edge weights $\beta_{ij}$) is said to be \emph{$(C,\eta)$-upper $L^p$ regular} if
$\alpha_i \leq \eta \alpha_G$ for all $i \in V(G)$, and whenever $V_1 \cup \dots \cup V_m$ is a partition of $V(G)$ into
disjoint vertex sets with $\alpha_{V_i} \geq \eta \alpha_G$ for each $i$, one has
\begin{equation}\label{eq:G-Lp}
\sum_{i,j=1}^m \frac{\alpha_{V_i}\alpha_{V_j}}{\alpha_G^2} \abs{\frac{\rho_G(V_i,V_j)}{\snorm{G}_1}}^p \leq C^p.
\end{equation}
\end{definition}

Informally, a graph $G$ is $(C,\eta)$-upper $L^p$ regular if $G /
\norm{G}_1$ has $L^p$ norm at most $C$ after we average over any partition
of the vertices into blocks of at least $\eta\abs{V(G)}$ in size (and no
vertex has weight greater than $\eta \alpha_G$).  We allow $p=\infty$, in
which case \eqref{eq:G-Lp} must be modified in the usual way to
\[
\max_{1\le i,j\le m} \abs{\frac{\rho_G(V_i,V_j)}{\snorm{G}_1}} \leq C.
\]
Strictly speaking, we should move $\snorm{G}_1$ to the right side of this
inequality and \eqref{eq:G-Lp}, to avoid possibly dividing by zero, but we
feel writing it this way makes the connection with $G/\snorm{G}_1$ clearer.

We will use the terms \emph{upper $L^p$ regular} and \emph{$L^p$ upper
regular} interchangeably. The former is used so that we do not end up writing
\emph{$(C, \eta)$ $L^p$ upper regular}, which looks a bit odd.

Note that the definition of $L^p$ upper regularity is interesting only for
$p > 1$, since \eqref{eq:G-Lp} automatically holds when $p = 1$ and $C = 1$.
See Appendix~\ref{sec:uniform} for a more refined definition, which plays
the same role when $p=1$.

Previous works on regularity and graph limits for sparse graphs (e.g.,
\cite{BR, Koh}) assume a strong hypothesis, namely that $|\rho_G(S,T)| \leq C
\snorm{G}_1$ whenever $\abs{S},\abs{T} \geq \eta\abs{V(G)}$. This is
equivalent to what we call $(C,\eta)$-upper $L^\infty$ regularity, and it is
strictly stronger than $L^p$ upper regularity for $p<\infty$. The
relationship between these notions will be come clearer when we discuss the
graph limits in a moment. For now, it suffices to say that the limit of a
sequence of $L^p$ upper regular graphs is a graphon with a finite $L^p$ norm.

\subsection{Graphons} \label{sec:defres-graphon}

In this paper, we define the term \emph{graphon} as follows.

\begin{definition} A \emph{graphon} is a symmetric, integrable function
$W \colon [0,1]^2 \to \RR$.
\end{definition}

Here \emph{symmetric} means $W(x,y) = W(y,x)$ for all $x,y \in [0,1]$. We
will use $\lambda$ to denote Lebesgue measure throughout this paper (on
$[0,1]$, $[0,1]^2$, or elsewhere), and \emph{measurable} will mean Borel
measurable.

Note that in other books and papers, such as \cite{BCLSV1,BCLSV2,Lov}, the
word ``graphon'' sometimes requires the image of $W$ to be in $[0,1]$, and
the term \emph{kernel} is then used to describe more general functions.

We define the $L^p$ norm on graphons for $1 \leq p < \infty$ by
\[
\norm{W}_p := (\EE [\abs{W}^p])^{1/p} = \paren{\int_{[0,1]^2}
  \abs{W(x,y)}^p \, dx\, dy}^{1/p},
\]
and $\norm{W}_\infty$ is the essential
supremum of $W$.

\begin{definition} An \emph{$L^p$ graphon} is a graphon $W$ with
$\norm{W}_p < \infty$.
\end{definition}

By nesting of norms, an $L^q$ graphon is automatically an $L^p$ graphon for
$1 \leq p \leq q \leq \infty$.  Note that as part of the definition, we
assumed all graphons are $L^1$.

We define the inner product for graphons by
\[
\ang{U, W} = \EE[UW] = \int_{[0,1]^2} U(x,y)W(x,y) \,dx\,dy.
\]
H\"older's inequality will be very useful:
\[
\abs{\ang{U, W}} \leq \norm{U}_p \norm{W}_{p'},
\]
where $1/p + 1/p' = 1$ and $1 \leq p, p' \leq \infty$. The special case
$p=p'=2$ is the Cauchy-Schwarz inequality.

Every weighted graph $G$ has an associated graphon $W^G$ constructed as
follows. First divide the interval $[0,1]$ into intervals $I_1, \dots,
I_{\abs{V(G)}}$ of lengths $\lambda(I_i) = \alpha_i / \alpha_G$ for each $i
\in V(G)$. The function $W^G$ is then given the constant value $\beta_{ij}$
on $I_i \x I_j$ for all $i, j \in V(G)$. Note that $\norm{W^G}_p =
\norm{G}_p$ for $1 \leq p \leq \infty$.

In the theory of dense graph limits, one proceeds by analyzing the
associated graphons $W^{G_n}$ for a sequence of graphs $G_n$, and in
particular one is interested in the limit of $W^{G_n}$ under the cut
metric. However, for sparse graphs, where the density of the graphs
tend to zero, the sequence $W^{G_n}$ converges to an uninteresting
limit of zero. In order to have a more interesting theory of sparse
graph limits, we consider the normalized associated graphons
$W^G/\norm{G}_1$ instead.

\begin{definition}[Stepping operator]
For a graphon $W \colon [0,1]^2 \to \RR$ and a partition $\cP = \{J_1, \dots,
J_m\}$ of $[0,1]$ into measurable subsets, we define a step-function $W_\cP
\colon [0,1]^2 \to \RR$ by
\[
W_\cP(x,y) := \frac{1}{\lambda(J_i)\lambda(J_j)} \int_{J_i \x J_j}
W \, d\lambda \qquad \text{for all } (x,y) \in J_i \x J_j.
\]
In other words, $W_\cP$ is produced from $W$
by averaging over each cell $J_i \x J_j$.
\end{definition}

A simple yet useful property of the stepping operator is that it is
contractive with respect to the cut norm $\norm{\cdot}_\square$
(defined in the next subsection) and all $L^p$ norms, i.e.,
$\norm{W_\cP}_\square \leq \norm{W}_\square$ and $\norm{W_\cP}_p \leq
\norm{W}_p$ for all graphon $W$ and $1 \leq p \leq \infty$.

We can rephrase the definition of a $(C,\eta)$-upper $L^p$ regular graph
using the language of graphons. Let $ V_1 \cup \cdots \cup V_m$ be a
partition of $V(G)$ as in Definition~\ref{def:Lp-upper-reg}, and let $\cP =
\{J_1, \dots, J_m\}$, where $J_i$ is the subset of $[0,1]$ corresponding to
$V_i$, i.e., $J_i = \bigcup_{v \in V_i} I_v$, where $I_v$ is as in the
definition of $W^G$. Then \eqref{eq:G-Lp} simply says that
\[
\snorm{(W^G)_\cP}_p \leq C \snorm{G}_1.
\]
This motivates the following notation of $L^p$ upper regularity for graphons.

\begin{definition} \label{def:W-upper-reg}
We say that a graphon $W \colon [0,1]^2 \to \RR$ is \emph{$(C,\eta)$-upper
$L^p$ regular} if whenever $\cP$ is a partition of $[0,1]$ into measurable
sets each having measure at least $\eta$,
\[
\norm{W_\cP}_p \leq C.
\]
\end{definition}

Given a weighted graph $G$, if the normalized associated graphon $W^G /
\norm{G}_1$ is $(C, \eta)$-upper $L^p$ regular and the vertex weights are all
at most $\eta \alpha_G$, then $G$ must also be $(C,\eta)$-upper $L^p$
regular. The converse is not true, as the definition of upper regularity for
graphons involves partitions $\cP$ of $[0,1]$ that do not necessarily respect
the vertex-atomicity of $V(G)$. For example, $K_3$ is a $(C, 1/2)$-upper
$L^p$ regular graph for every $C>0$ and $p > 1$ because no valid partition of
vertices exist, but the same is not true for the graphon
$W^{K_3}/\norm{K_3}_1$.

\subsection{Cut metric} \label{sec:defres-cut}

The most important metric on the space of graphons is the cut metric.
(Strictly speaking, it is merely a pseudometric, since two graphons with cut
distance zero between them need not be equal.)  It is defined in terms of the
cut norm introduced by Frieze and Kannan \cite{FK}.

\begin{definition}[Cut metric] For a graphon
$W \colon [0,1]^2 \to \RR$, define the \emph{cut norm} by
\begin{equation} \label{eq:cut-norm}
\norm{W}_\square := \sup_{S,T \subseteq [0,1]} \abs{\int_{S \x T}
W(x,y) \, dx\,dy},
\end{equation}
where $S$ and $T$ range over measurable subsets of $[0,1]$. Given two
graphons $W, W' \colon [0,1]^2 \to \RR$, define
\[
d_\square(W,W') := \norm{W - W'}_\square
\]
and the \emph{cut metric} (or \emph{cut distance}) $\delta_\square$ by
\[
\delta_\square(W,W') := \inf_\sigma d_\square(W^\sigma,W'),
\]
where $\sigma$ ranges over all measure-preserving bijections $[0,1] \to
[0,1]$ and $W^\sigma(x,y) := W(\sigma(x), \sigma(y))$.
\end{definition}

For a survey covering many properties of the cut metric, see \cite{Jan}.  One
convenient reformulation is that it is equivalent to the $L^\infty\to L^1$
operator norm, which is defined by
\[
\norm{W}_{\infty \to 1} = \sup_{\norm{f}_\infty,\norm{g}_\infty \le 1}
\abs{\int_{[0,1]^2} W(x,y) f(x) g(y) \, dx \,dy},
\]
where $f$ and $g$ are functions from $[0,1]$ to $\RR$.  Specifically, it is
not hard to show that
\begin{equation} \label{eq:cut-to-operator-norm}
\norm{W}_\square \le \norm{W}_{\infty \to 1} \le 4 \norm{W}_\square,
\end{equation}
by checking that $f$ and $g$ take on only the values $\pm 1$ in the extreme
case.

We can extend the $d$ and $\delta$ notations to any norm on the space of
graphons. In particular, for $1 \leq p \leq \infty$, we define
\[
d_p(W, W') := \norm{W - W'}_p
\qquad
\text{and}
\qquad
\delta_p(W,W') := \inf_\sigma d_p(W^\sigma,W'),
\]
with $\sigma$ ranging overall measure-preserving bijections $[0,1] \to
[0,1]$ as before.

To define the cut distance between two weighted graphs $G$ and $G'$, we use
their associated graphons. If $G$ and $G'$ are weighted graphs on the same
set of vertices (with the same vertex weights), with edge weights given by
$\beta_{ij}(G)$ and $\beta_{ij}(G')$ respectively, then we define
\[
d_\square(G, G') := d_\square(W^G, W^{G'})
= \max_{S,T \subseteq V(G)} \abs{\sum_{i \in S, j \in T} \frac{\alpha_i
    \alpha_j}{\alpha_G^2} (\beta_{st}(G) - \beta_{st}(G'))},
\]
where $W^G$ and $W^{G'}$ are constructed using the same partition of $[0,1]$
based on the vertex set. The final equality uses the fact that the cut norm
for a graphon associated to a weighted graph can always be achieved by $S$
and $T$ in \eqref{eq:cut-norm} that correspond to vertex subsets. This is due
to the bilinearity of the expression of inside the absolute value in
\eqref{eq:cut-norm} with respect to the fractional contribution of each
vertex to the sets $S$ and $T$.

When $G$ and $G'$ have different
vertex sets, $d_\square(G, G')$ no longer makes sense, but it still
makes sense to define
\[
\delta_\square(G, G') := \delta_\square(W^G, W^{G'}).
\]
Similarly, for a weighted graph $G$ and a graphon $W$, define
\[
\delta_\square(G, W) := \delta_\square(W^G, W).
\]
To compare graphs of different densities, we can compare the normalized
associated graphons, i.e., $\delta_\square(G/\snorm{G}_1,
G'/\snorm{G'}_1)$. We will sometimes refer to this quantity as the
\emph{normalized cut metric}.

\subsection{$L^p$ upper regular sequences} \label{sec:defres-up-reg-seq}

\begin{definition} Let $1 < p \leq \infty$ and $C > 0$. We say that
  $(G_n)_{n \ge 0}$ is \emph{a $C$-upper $L^p$ regular sequence} of weighted
  graphs if for every $\eta > 0$ there is some $n_0 = n_0(\eta)$ such
  that $G_n$ is $(C + \eta, \eta)$-upper $L^p$ regular for all $n \geq
  n_0$. In other words, $G_n$ is $(C + o(1), o(1))$-upper $L^p$
  regular as $n \to \infty$. An $L^p$ upper regular sequence of graphons is defined
  similarly.
\end{definition}

As an example what kind of graphs this definition excludes, a sequence of
graphs $G_n$ formed by taking a clique on a subset of $o(\abs{V(G_n)})$
vertices and no other edges is not $C$-upper $L^p$ regular for any $1 < p
\leq \infty$ and $C > 0$.  Furthermore, in Appendix~\ref{sec:superlinear} we
show that the average degree in a $C$-upper $L^p$ regular sequence of simple
graphs must tend to infinity.

Now we are ready to state one of the main results of the paper, which
asserts the existence of limits for $L^p$ upper regular sequences.

\begin{theorem} \label{thm:G-limit}
Let $p > 1$ and let $(G_n)_{n \ge 0}$ be a $C$-upper $L^p$ regular sequence of weighted graphs. Then
there exists an $L^p$ graphon $W$ with $\norm{W}_p \leq C$ so that
\[
\liminf_{n\to\infty} \delta_\square\paren{\frac{G_n}{\snorm{G_n}_1}, W}= 0.
\]
\end{theorem}

In other words, some subsequence of $G_n/\snorm{G_n}_1$ converges to
$W$ in the cut metric.  An analogous result holds for $L^p$ upper
regular sequences of graphons.

\begin{theorem} \label{thm:W-upper-reg-limit}
Let $p > 1$ and let $(W_n)_{n \ge 0}$ be a $C$-upper $L^p$ regular sequence of graphons. Then
there exists an $L^p$ graphon $W$ with $\norm{W}_p \leq C$ so that
\[
\liminf_{n\to\infty} \delta_\square(W_n, W) = 0.
\]
\end{theorem}

These theorems, and all the remaining results in this subsection, are proved
in \textsection\ref{sec:limit}.

The next proposition says that, conversely, every sequence that
converges to an $L^p$ graphon must be an $L^p$ upper regular
sequence.

\begin{proposition} \label{prop:conv-to-Lp} Let $1 \leq p \leq \infty$, let
  $W$ be an $L^p$ graphon, and let $(W_n)_{n \ge 0}$ be a sequence of graphons with
  $\delta_\square(W_n, W) \to 0$ as $n \to \infty$. Then $(W_n)_{n \ge 0}$ is a
  $\norm{W}_p$-upper $L^p$ regular sequence.
\end{proposition}

An analogous result about weighted graphs follows as an immediate
corollary by setting $W_n = W^{G_n}/\norm{G_n}_1$.

\begin{corollary} \label{cor:G-conv-to-Lp} Let $1 \leq p \leq
  \infty$, let $W$ be an $L^p$ graphon, and let $(G_n)_{n \ge 0}$ be a sequence of
  weighted graphs with no dominant nodes and with $\delta_\square(G_n/\norm{G_n}_1, W) \to 0$ as $n \to
  \infty$. Then $(G_n)_{n \ge 0}$ is a $\norm{W}_p$-upper $L^p$ regular
  sequence.
\end{corollary}

The two limit results, Theorems~\ref{thm:G-limit} and
\ref{thm:W-upper-reg-limit}, are proved by first developing a
regularity lemma showing that one can approximate an $L^p$ upper
regular graph(on) by an $L^p$ graphon with respect to cut metric, and
then establishing a limit result in the space of $L^p$ graphons. The
latter step can be rephrased as a compactness result for $L^p$
graphons, which we state in the next subsection.

We note that a sequence of graphs might not have a limit without the $L^p$
upper regularity assumption. It could go wrong in two ways: (a) a sequence
might not have any Cauchy subsequence, and (b) even a Cauchy sequence is not
guaranteed to converge to a limit.

\begin{proposition}   \label{prop:G-no-limit}
  (a) There exists a sequence of simple graphs
    $G_n$ so that
    \[
    \delta_\square(G_n/\norm{G_n}_1, G_m/\norm{G_m}_1) \geq 1/2 \qquad
    \text{for all $n$ and $m$ with $n \neq m$}.
    \]

    \noindent (b) There exists a sequence of simple graphs $G_n$ such that
    $(G_n/\norm{G_n}_1)_{n \ge 0}$ is a Cauchy sequence with respect to
    $\delta_\square$ but does not converge to any graphon $W$ with
    respect to $\delta_\square$.
\end{proposition}

\subsection{Compactness of $L^p$ graphons} \label{sec:defres-compact}

Lov\'asz and Szegedy~\cite{LS07} proved that the space of $[0,1]$-valued
graphons is compact with respect to the cut distance (after
identifying graphons with cut distance zero). We extend this result to
$L^p$ graphons.

\begin{theorem}[Compactness of the $L^p$ ball with respect to cut
  metric] \label{thm:compact} Let $1 < p \leq \infty$ and $C > 0$, and let
  $(W_n)_{n \ge 0}$ be a sequence of $L^p$ graphons with $\norm{W_n}_p\leq C$
  for all $n$. Then there exists an $L^p$ graphon $W$ with $\norm{W}_p
  \leq C$ so that
  \[
  \liminf_{n \to
    \infty} \delta_\square(W_n, W) = 0.
  \]
  In other words, $\cB_{L^p}(C) := \{L^p \text{ graphons } W:
  \norm{W}_p \leq C\}$ is compact with respect to the cut metric
  $\delta_\square$ (after identifying points of distance zero).
\end{theorem}

For a proof, see \textsection\ref{sec:Lp-graphons}. The analogous claim for
$p=1$ is false without additional hypotheses, as
Proposition~\ref{prop:G-no-limit} implies that the $L^1$ ball of graphons is
neither totally bounded nor complete with respect to $\delta_\square$. The
example showing that the $L^1$ ball is not totally bounded is easy: the
sequence $W_n = 2^{2n} 1_{[2^{-n},2^{-n}]\times [2^{-n},2^{-n}]}$ satisfies
$\delta_\square(W_n, W_m) > 1/2$ for every $m \neq n$. Our example showing
incompleteness is a bit more involved, and we defer it to the proof of
Proposition~\ref{prop:G-no-limit}(b).  See Theorem~\ref{thm:unifintcompact}
for an $L^1$ version of Theorem~\ref{thm:compact} under the hypothesis of
uniform integrability.

\subsection{Sparse $W$-random graph models} \label{sec:defres-W-random}

Our main result on this topic is that every graphon $W$ gives rise to a
natural random graph model, which produces a sequence of sparse graphs
converging to $W$ in the normalized cut metric. When $W$ is nonnegative, the
model produces sparse simple graphs. If $W$ is allowed negative values, the
resulting random graphs have $\pm 1$ edge weights.

We explain this construction in two steps.

\medskip 

\emph{Step 1: From $W$ to a random weighted graph.} Given any graphon
$W$, define $\bH(n,W)$ to be a random weighted
graph on $n$ vertices (labeled by $[n] = \{1, 2, \dots, n\}$, with all
vertex weights 1) constructed as follows: let $x_1, \dots, x_n$ be
i.i.d.\ chosen uniformly in $[0,1]$, and then assign the weight of the
edge $ij$ to be $W(x_i, x_j)$ for all distinct $i, j \in [n]$.

\emph{Step 2: From a weighted graph to a sparse random graph.}
Let $H$ be a weighted graph with $V(H) = [n]$ (with all vertex weights
1) and edge weights $\beta_{ij}$ (with $\beta_{ii} = 0$), and let
$\rho > 0$. When $\beta_{ij} \geq 0$ for all $ij$, the sparse random simple graph
$\bG(H,\rho)$ is defined by taking $V(H)$ to be the set of vertices
and letting $ij$ be an edge with probability $\min\{\rho \beta_{ij},
1\}$, independently for all $ij \in E(H)$. If we allow negative edge
weights on $H$, then we take $\bG(H,\rho)$ to be a random graph
with edge weights $\pm 1$, where $ij$ is made an edge with probability
$\min\{\rho \sabs{\beta_{ij}}, 1\}$ and given edge weight $+1$ if
$\beta_{ij} > 0$ and $-1$ if $\beta_{ij} < 0$.

Finally, given any graphon $W$ we define the sparse
$W$-random (weighted) graph to be $\bG(n, W, \rho) :=
\bG(\bH(n,W),\rho)$.

\medskip 

We also view $\bH(n,W)$ and $\bG(n,W,\rho_n)$ as graphons in the usual way,
where the vertices are
ordered according to the ordering of $x_1, \dots, x_n$ as real
numbers and each vertex is represented by an interval of length $1/n$. For example, we use this interpretation in the notation
$d_1(\bH(n,W), W)$.

\medskip 

Note that it is also possible to consider other random weighted graph models
where the edge weights are chosen from some other distribution (other than
$\pm 1$). Many of our results generalize easily, but we stick to our model
for simplicity.

Here is our main theorem on $W$-random graphs. Note that we use the same
i.i.d.\ sequence $x_1, x_2, \dots$ for constructing $\bH(n,W)$ and $\bG(n,W,\rho_n)$
for different values of $n$, i.e., without resampling the $x_i$'s.

\begin{theorem}[Convergence of $W$-random graphs] \label{thm:rand-converge}
  Let $W$ be an $L^1$ graphon.
  \begin{enumerate}
  \item[(a)]  We have $d_1(\bH(n,W), W) \to
    0$ as $n\to \infty$ with probability $1$.
  \item[(b)] If $\rho_n >0 $ satisfies $\rho_n \to 0$ and $n
    \rho_n \to \infty$ as $n \to \infty$, then
    \[
    d_\square(\rho_n^{-1}\bG(n,W,\rho_n), W) \to 0
    \]
    as $n \to \infty$ with probability $1$.
  \end{enumerate}
\end{theorem}

Part (a) is proved in \textsection\ref{sec:weighted-random} and part (b) in
\textsection\ref{sec:sparse-random}.  Note that we use $d_1$ and $d_\square$
(as opposed to $\delta_1$ and $\delta_\square$) because we have ordered the
vertices of the graphs according to the ordering of the sample points
$x_1,\dots,x_n$. Of course the sample point ordering is not determined by the
graphs alone.

\begin{corollary} \label{cor:G-rand-converge-normalized}
  Let $W$ be an $L^1$ graphon with $\norm{W}_1 > 0$. Let $\rho_n >0$ satisfy $\rho_n \to 0$ and $n
    \rho_n \to \infty$ as $n \to \infty$, and let $G_n = \bG(n,
  W, \rho_n)$. Then
  \[
  \delta_\square(G_n / \norm{G_n}_1, W /
  \norm{W}_1) \to 0
  \]
  as $n \to \infty$ with probability $1$.
\end{corollary}

Furthermore, for any $1 \leq p \leq \infty$, if $W$ is an $L^p$ graphon, then
$\norm{\bH(n,W)}_p \to \norm{W}_p$ with probability $1$ (this is an immediate
consequence of Theorem~\ref{thm:SLLN-W} below).  Thus $\bH(n,W)$ generates a
sequence of $L^p$ graphons converging to $W$. Also, by
Proposition~\ref{prop:conv-to-Lp} and Theorem~\ref{thm:rand-converge}(b),
$\bG(n,W,\rho_n)$ is a $\norm{W}_p$-upper $L^p$ regular sequence that
converges to $W$ in normalized cut metric.

Note that the sparsity assumption $\rho_n \to 0$ is necessary since the
edges of $\bG(n,W,\rho_n)$ are included with probability $\min\{\rho_n
\abs{W(\cdot,\cdot)}, 1\}$, so $\rho_n$ needs to be arbitrarily close to
zero in order to ``see'' the unbounded part of $W$.  Similarly, the
assumption that $n \rho_n \to \infty$ means the expected average degree
tends to infinity, which is necessary by Corollary~\ref{cor:G-conv-to-Lp}
and Proposition~\ref{prop:superlinear}.

We will prove Theorem~\ref{thm:rand-converge}(a) using a theorem of Hoeffding on $U$-statistics,
while Theorem~\ref{thm:rand-converge}(b) follows from
Theorem~\ref{thm:rand-converge}(a) via a Chernoff-type argument that
shows that if $H$ is a weighted graph with many vertices, then $\rho^{-1}\bG(H, \rho)$ is
close to $H$ in cut metric.

Theorem~\ref{thm:rand-converge} was proved for $L^\infty$ graphons
as Theorem~4.5 in \cite{BCLSV1},\footnote{Technically, Theorem~4.5 in \cite{BCLSV1}
is just a close analogue, since it uses $\delta_\square$ instead of $d_1$ and $d_\square$.}
but the proof given there does not seem to extend to
Theorem~\ref{thm:rand-converge}. The proof here is much shorter
than that in \cite{BCLSV1}, though, unlike that proof, our proof
gives no quantitative guarantees.

Using sparse $W$-random graphs, we can fully justify the name
$W$-\emph{quasirandom} for graphs approximating a graphon $W$.  The following
proposition shows that every sequence of sparse simple graphs converging to
$W$ is close in cut metric to $W$-random graphs:

\begin{proposition} \label{prop:Wquasirandom}
  Let $p > 1$, and let $(G_n)_{n \ge 0}$ be a
  sequence of simple graphs such that $\norm{G_n}_1 \to 0$ and
  $\delta_\square(G_n/\norm{G_n}_1, W) \to 0$, where $W$ is an $L^p$ graphon.
  Let $G'_n = \bG(\abs{V(G_n)}, W, \norm{G_n}_1)$. Then with probability $1$,
  one can order the vertices of $G_n$ and $G'_n$ so that
  \[
  d_\square\paren{\frac{G_n}{\norm{G_n}_1}, \frac{G'_n}{\norm{G'_n}_1}} \to 0.
  \]
\end{proposition}

See \textsection\ref{sec:sparse-random} for the proof, and
Proposition~\ref{prop:Wquasirandomunif} for a generalization to $p=1$.

\subsection{From upper regular sequences to graphons and back} \label{sec:defres-diagram}

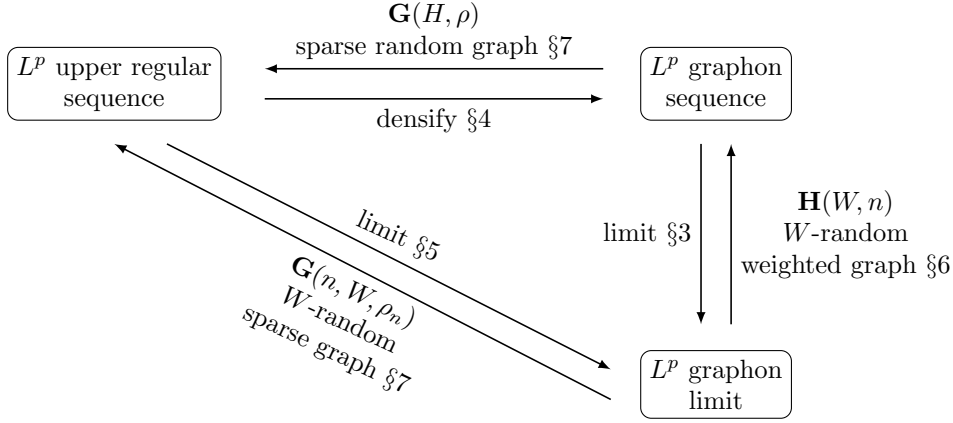
\begin{figure} \centering
\begin{tikzpicture}
  \node[draw,rounded corners,align=center] (A) at (-4,4) {$L^p$ upper regular \\ sequence};
  \node[draw,rounded corners,align=center] (B) at (4,4) {$L^p$ graphon \\ sequence};
  \node[draw,rounded corners,align=center] (C) at (4,0) {$L^p$ graphon \\ limit };

  \draw[latex-,semithick] (-2,4.2) -- node[above,align=center]
  {$\bG(H,\rho)$ \\ sparse random graph \S\ref{sec:sparse-random}} (2.5,4.2);
  \draw[-latex,semithick] (-2,3.8) -- node[below] {densify \S\ref{sec:reg-lem-up-reg}} (2.5,3.8);
  \draw[latex-,semithick] (4.2,3.2) -- node[right,align=center]
  {$\bH(W,n)$ \\ $W$-random \\ weighted graph \S\ref{sec:weighted-random}} (4.2,.8);
  \draw[-latex,semithick] (3.8,3.2) -- node[left] {limit \S\ref{sec:Lp-graphons}}
(3.8,.8);
  \draw[-latex,semithick] (2.6,-.2) -- node[below,rotate=-25,align=center]
  {$\bG(n,W,\rho_n)$ \\ $W$-random \\ sparse graph \S\ref{sec:sparse-random}} (-4,3.2);
  \draw[latex-,semithick] (2.6,.2) -- node[above,rotate=-25] {limit
  \S\ref{sec:limit}} (-3.3,3.2);
\end{tikzpicture}
\caption{The relationships between the objects studied in this
  paper. The arrows are labeled with the relevant sections.} \label{fig:summary}
\end{figure}

In Figure~\ref{fig:summary} we summarize the relationship between the
objects studied in this paper. The inner set of arrows describe the
process of going from a sequence to a limit, while the outer arrows
describe the process of starting from a graphon $W$ and constructing a
sequence via a $W$-random graph model. Although we are primarily
interested in the diagonal arrows connecting $L^p$ upper regular
sequences and $L^p$ graphon limits, the proofs, in both directions, go
through $L^p$ graphons as a useful intermediate step.

We have not yet discussed the term \emph{densify} in
Figure~\ref{fig:summary}. By densifying we mean approximating (in the sense
of cut distance) an $L^p$ upper regular graph by an $L^p$ graphon. The former
can be thought of as a sequence of sparse graphs with large edge weights
supported on a sparse set of edges (although they do not have to be), and the
latter as graphs on a dense set of edges with small weights (in the sense of
being $L^p$ bounded). More precisely, we prove the following result, which we
think of as a transference theorem in the spirit of Green and Tao \cite{GT}.

\begin{proposition} \label{prop:densify}
  For every $p > 1$ and $\e > 0$ there exists an $\eta > 0$ such that for
  every $(C,\eta)$-upper $L^p$ regular weighted graph $G$ (or
  graphon $W$), there exists an $L^p$ graphon $U$ with $\norm{U}_p
  \leq C$ such that
  \[
  \delta_\square\paren{\frac{G}{\snorm{G}_1}, U} \leq C \e \qquad
  \text{(respectively, }
  \delta_\square(W, U) \leq C \e\text{).}
  \]
\end{proposition}

We establish Proposition~\ref{prop:densify} as a weak regularity lemma. In
fact, $U$ can be constructed from $G$ by averaging the edge weights over a
partition of the vertex set of $G$. As with other regularity lemmas, the
number of parts used in the partition will be bounded.  See
\textsection\ref{sec:reg-lem-up-reg} for the proof.

The regularity lemma for dense graphs was developed by Szemer\'edi~\cite{Sz}.
Extensions of Szemer\'edi's regularity lemma to sparse graphs were developed
independently by Kohayakawa and R\"odl~\cite{Koh,KR} under an $L^\infty$
upper regularity assumption. Scott~\cite{Sco} gave another proof of a sparse
regularity lemma without any assumptions, but as in Szemer\'edi's regularity
lemma, it allows for exceptional parts that could potentially hide all the
``dense spots.'' Frieze and Kannan~\cite{FK} developed a weak version of
regularity lemma with better bounds on the number of parts needed, and it is
the version that we extend. This weak regularity lemma was extended to sparse
graphs under the $L^\infty$ upper regularity assumption in~\cite{BR} and
\cite{COCF}. In our work, we extend the weak regularity lemma to $L^p$ upper
regular graphs.

Our proof of the weak regularity lemma for $L^p$ upper regular graphs is an
extension of the usual $L^2$ energy increment argument. However, the
extension is not completely straightforward. Due to the nesting of norms,
when $1 < p < 2$, we do not have very much control over the maximum $L^2$
energy for an $L^p$ upper regular graph. This issue does not arise when $p
\geq 2$ (e.g., $p = \infty$ in previous works). We resolve this issue via a
careful truncation argument when $1 < p < 2$.  As it turns out, these
truncation arguments can be generalized to the case $p=1$, provided we have
sufficient control over the tails of $W$; see Appendix~\ref{sec:uniform}.

\subsection{Counting lemma for $L^p$ graphons} \label{sec:defres-count}

We have not yet addressed the issue of subgraph counts.\footnote{We
  actually only talk about homomorphism counts in this paper. There is
  a subtle yet significant distinction between homomorphisms and
  subgraphs, namely that subgraphs arise as homomorphisms for which the
  map $V(F) \to V(G)$ is injective. When $G$ is a large, dense graph and $F$ is fixed, this
  distinction is not important, since all but a vanishing proportion
  of maps $V(F) \to V(G)$ are injective. However, when $G$ is sparse,
  this distinction could be significant (since the normalization is to
  divide the subgraph count by $\norm{G}_1^{\abs{E(F)}}
  \abs{V(G)}^{\abs{V(F)}}$). As an example, when $\rho =
  o(n^{-1/2})$, we have $n^4 \rho^4 = o(n^3 \rho^2)$, so the main contribution
  to the number of homomorphisms from $C_4$ to the random graph $G(n,
  \rho)$ is no longer coming from 4-cycles, but rather from paths of
  length $2$ (each of which is the image of a homomorphism from
  $C_4$). However, as it turns out, we will not say much about
  either homomorphism densities or subgraph counts for sparse graphs
  anyway (our counting lemmas are for $L^p$ graphons), so let us not
  dwell on the distinction between subgraphs and homomorphisms.}
For simple graphs $F$ and $G$, a graph homomorphism from $F$ to $G$ is a map
$V(F) \to V(G)$ that sends every edge of $F$ to an edge of $G$. Let $\hom(F,
G)$ be the number of homomorphisms. The homomorphism density, or
\emph{$F$-density}, is defined by $t(F, G) := \hom(F,
G)/\abs{V(G)}^{\abs{V(F)}}$, which is equal to the probability that a random
map $V(F) \to V(G)$ is a homomorphism.

In the theory of dense graph limits, the importance of homomorphism densities
is that they characterize convergence under the cut metric: a sequence of
dense graphs converges if and only if its $F$-densities converge for all $F$,
and the limiting $F$-densities then describe the resulting graphon
\cite[Theorem~3.8]{BCLSV1}. This notion of convergence is called \emph{left
convergence}.

The situation is decidedly different for sparse graphs, and left convergence
is not even implied by cut metric convergence, as we will see below.  The
irrelevance of left convergence is the most striking difference between dense
and sparse graph limits, and it is an unavoidable consequence of sparsity. By
contrast, right convergence (defined by quotients or statistical physics
models) remains equivalent to metric convergence, as we show in \cite{BCCZ}.

Before explaining further, we must extend the definition of homomorphism
density to weighted graphs and graphons. For any simple graph $F$ and graphon
$W$, we define
\[
t(F,W) := \int_{[0,1]^{\abs{V(F)}}} \prod_{ij \in E(F)} W(x_i, x_j)
\,dx_1\dots dx_{\abs{V(F)}}.
\]
Note that $t(F, G) = t(F, W^G)$ for simple graphs $G$, and we take
this as the definition of $t(F, G)$ for weighted graphs $G$.

A \emph{counting lemma} is a claim that any two graphs/graphons that are
close in cut metric must have similar $F$-densities. For dense graphs (or
more generally, graphs with uniformly bounded edge weights), this claim is
not too hard to show. For example, the following counting lemma appears in
\cite[Theorem~3.7(a)]{BCLSV1}.

\begin{theorem}[Counting lemma for $L^\infty$ graphons] \label{thm:L^infty-counting}
  Let $F$ be a simple graph with $m$ edges. If $U$ and $W$ are graphons with
  $\norm{U}_\infty \leq 1$, $\norm{W}_\infty \leq 1$, and $\delta_\square(U, W)
  \leq \e$, then
  \[
  \abs{t(F,U) - t(F,W)} \leq 4m \e.
  \]
\end{theorem}

However, for sparse graphs, a general counting lemma of this form is too much
to ask for, even for $L^\infty$ upper regular graphs. Here is an example
illustrating this difficulty. Let $G_n$ be an instance of the Erd\H
os-R\'enyi random graph $G(n, \rho_n)$, where $\rho_n > 0$ is the edge
probability. If $n \rho_n \to \infty$, then $\rho_n^{-3} t(K_3, G_n) \to 1$
by a standard second moment argument, e.g., \cite[Theorem~4.4.4]{AS}. Let
$G'_n$ be obtained from $G_n$ by deleting edges from all triangles in $G_n$.
If we additionally assume $\rho_n = o(n^{-1/2})$, so that $n^3 \rho_n^3 =
o(n^2 \rho_n)$ and hence only an $o(1)$ fraction of the edges of $G_n$ are
deleted, then $d_\square(\rho_n^{-1} G_n, \rho_n^{-1} G'_n) = o(1)$. It
follows that $G_n$ and $G'_n$ are close in (normalized) cut distance, but
have very different (normalized) triangle densities, as $t(K_3, G_n') = 0$.
This example shows that we cannot expect a general counting lemma even for
$L^\infty$ upper regular sparse graphs, let alone $L^p$ upper regular graphs.

Nevertheless, we will give a counting lemma for $L^p$ graphons (which is the
``dense setting,'' as opposed to the ``sparse setting'' of $L^p$ upper
regular graphons). There is already an initial difficulty, which is that
$t(F, W)$ might not be finite. The next proposition shows the conditions for
$t(F, W)$ to be finite.

\begin{proposition} \label{prop:count-finiteness}
  Let $F$ be a simple graph with maximum degree $\Delta$. For
    every $p < \Delta$, there exists an $L^p$ graphon $W$ with $t(F, W)
    = \infty$. On the other hand, if $W$ is an $L^\Delta$ graphon,
    then $t(F, W)$ is well-defined and finite. Furthermore, $\abs{t(F, W)} \leq \norm{W}_\Delta^{\abs{E(F)}}$.
\end{proposition}

We want a counting lemma which asserts that if $U$ and $W$ are graphons with
bounded $L^p$ norms, then $\abs{t(F, U) - t(F, W)}$ is small whenever
$\delta_\square(U, W)$ is small. Proposition~\ref{prop:count-finiteness}
suggests we should not expect such a counting lemma to hold when $p <
\Delta$. In fact, we give a counting lemma whenever $p > \Delta$ and show
that no counting lemma can hold when $p \leq \Delta$.

We prove the following extension of Theorem~\ref{thm:L^infty-counting} to
$L^p$ graphons. Note that for fixed $F$ and $p$, the bound in
\eqref{eq:Lp-counting} is a function of $\e$ that goes to zero as $\e \to 0$.
As $p \to \infty$, the bound in Theorem~\ref{thm:Lp-counting} converges to
that of Theorem~\ref{thm:L^infty-counting}.

\begin{theorem}[Counting lemma for $L^p$ graphons] \label{thm:Lp-counting}
Let $F$ be a simple graph with $m$ edges and maximum degree $\Delta$. Let
$\Delta < p < \infty$. If $U$ and $W$ are graphons with $\norm{U}_p \leq 1$,
$\norm{W}_p \leq 1$, and $\delta_\square(U, W) \leq \e$, then
\begin{equation} \label{eq:Lp-counting}
\abs{t(F,U) - t(F,W)} \leq
2m(m-1+p-\Delta) \paren{\frac{2\e}{p-\Delta}}^{\frac{p-\Delta}{p-\Delta+m-1}}.
\end{equation}
\end{theorem}

The counting lemma implies the following corollary for sequences of graphons
that are uniformly bounded in $L^p$ norm. As we saw above, $L^p$ upper
regularity would not suffice.

\begin{corollary} \label{cor:Lp-count-conv}
Let $p > 1$ and $C > 0$, and let $W_n$ be a sequence of graphons converging
to $W$ in cut metric. Suppose $\norm{W_n}_p \leq C$ for all $n$ and
$\norm{W}_p \leq C$. Then for every simple graph $F$ with maximum degree less
than $p$, we have $t(F, W_n) \to t(F, W)$ as $n \to \infty$.
\end{corollary}

On the other hand, no counting lemma can hold when $p \leq \Delta$, even if we replace
the cut norm by the $L^1$ norm.

\begin{proposition} \label{eq:Lp-no-count}
Let $F$ be a simple graph with maximum degree $\Delta \geq 2$, and let $1 \leq p
\leq \Delta$. Then there exists a sequence $(W_n)_{n \ge 0}$ of graphons with
$\norm{W_n}_p \leq 4$ such that $\norm{W_n - 1}_1 \to 0$ as $n \to \infty$
yet
\[
\lim_{n \to \infty} t(F, W_n) = 2^{\abs{\{v \in V(G) \,:\, \deg_F(v) = \Delta\}}} > 1 = t(F,1). 
\]
\end{proposition}

See \textsection\ref{sec:counting} for proofs of these results.

\section{$L^p$ graphons} \label{sec:Lp-graphons}

Recall that an $L^p$ graphon is a symmetric and integrable function $W \colon
[0,1]^2 \to \RR$ with $\norm{W}_p < \infty$. In this section, we prove
Theorem~\ref{thm:compact}, which gives a limit theorem for $L^p$ graphons.
The results in this section form the ($L^p$ graphon sequence) $\to$ ($L^p$
graphon limit) arrow in Figure~\ref{fig:summary}.

The proof technique is an extension of that of \cite{LS07}. We will need a
weak regularity lemma for $L^p$ graphons. The standard proof of the weak
regularity lemma involving $L^2$ energy increments, based on ideas from
\textsection{8} of \cite{FK}, works for $L^2$ graphons and hence $L^p$
graphons for $p \geq 2$. Since several of our proofs are based on the same
basic idea, we include the proof here.  When $1 < p < 2$, we use a
truncation argument to reduce to the $p=2$ case.

\begin{lemma}[Weak regularity lemma for $L^2$ graphons] \label{lem:W-L^2-reg}
Let $\e >0 $, let $W \colon [0,1]^2 \to \RR$ be an $L^2$
graphon, and let $\cP$ be a partition of
$[0,1]$. Then there exists a partition $\cQ$ refining $\cP$ into
at most $4^{1/\e^2} \abs{\cP}$ parts so that
\[
\norm{W - W_\cQ}_\square \leq \e \norm{W}_2.
\]
\end{lemma}

\begin{proof}
  We build a sequence $\cP_0, \cP_1, \cP_2, \dots$ of partitions of
  $[0,1]$, starting with $\cP_0 = \cP$. For each $i \geq 0$, the
  partition $\cP_{i+1}$ refines $\cP_i$ by dividing each part of
  $\cP_i$ into at most four subparts. So in particular $\abs{\cP_i}
  \leq 4^i \abs{\cP_0}$.

  These partitions are constructed as follows. If for some $i$,
  $\cP_i$ satisfies $\norm{W - W_{\cP_i}}_\square \leq \e\snorm{W}_2$,
  then we stop. Otherwise, by the definition of the cut norm, there
  exists measurable subsets $S,T \subseteq [0,1]$ with
  \[
    \abs{\ang{W - W_{\cP_i},1_{S
          \x T}}} > \e\norm{W}_2.
  \]
  Let $\cP_{i+1}$ be the common refinement of $\cP_i$ with $S$ and
  $T$. Since $S$ and $T$ are both unions of parts in $\cP_{i+1}$,
  \[
  \abs{\ang{W_{\cP_{i+1}} - W_{\cP_i},1_{S \x T}}} = \abs{\ang{W -
      W_{\cP_i},1_{S \x T}}} >  \e\norm{W}_2.
  \]
  Since $\cP_{i+1}$ is a
refinement of $\cP_i$, $\ang{W_{\cP_{i+1}} - W_{\cP_i}, W_{\cP_i}} = 0$. So
by the Pythagorean theorem, followed by the Cauchy-Schwarz inequality,
\[
\norm{W_{\cP_{i+1}}}_2^2 - \norm{W_{\cP_{i}}}_2^2
= \norm{W_{\cP_{i+1}} - W_{\cP_{i}}}_2^2
\geq \abs{\ang{W_{\cP_{i+1}} - W_{\cP_i},1_{S \x T}}}^2
> \e^2\norm{W}_2^2.
\]
Since $\snorm{W_{\cP_i}}_2 \leq \snorm{W}_2$ (by the convexity of $x
\mapsto x^2$), we see that the process must stop with $i \leq
1/\e^2$. The final $\cP_i$ is the desired $\cQ$.
\end{proof}

An \emph{equipartition} of $[0,1]$ is a partition where all parts have equal
measure. It will be convenient to enforce that the partitions obtained from
the regularity lemma are equipartitions. The following lemma is similar to
\cite[Lemma~9.15(b)]{Lov}.

\begin{lemma}[Equitizing a partition] \label{lem:equitize}
Let $p>1$ and $\e > 0$, and let $k$ be any positive integer. Let $W$ be an
$L^p$ graphon, let $\cP$ be an equipartition of $[0,1]$, and let $\cQ$ be a
partition refining $\cP$. Then there exists an equipartition $\cQ'$ refining
$\cP$ into exactly $k \abs{\cP}$ parts so that
\[
\norm{W - W_{\cQ'}}_\square \leq 2\norm{W - W_{\cQ}}_\square + 2 \norm{W}_p \paren{\frac{2\abs{\cQ}}{k\abs{\cP}}}^{1-1/p}.
\]
\end{lemma}

\begin{proof}
For $\cQ'$ we choose any equipartition refining $\cP$ into exactly
$k\abs{\cP}$ parts, at most $\abs{\cQ}$ of which intersect more than one part
of $\cQ$. We can construct such a $\cQ'$ as follows. For each part $P_i$ of
$\cP$, let $Q_{i1}, \dots, Q_{im}$ be the parts of $\cQ$ contained in $P_i$.
Form $\cQ'$ by dividing up each of $Q_{i1}, \dots, Q_{im}$ into parts of
measure exactly $1/(k\abs{\cP})$ plus a remainder part; then group the
remainder parts in $P_i$ together and divide them into parts of measure
$1/(k\abs{\cP})$. This partitions $P_i$ into $k$ parts of equal size.  At
most $m$ of these new parts intersect more than one part of $\cQ$, because
there were at most $m$ remainder parts, each of size less than
$1/(k\abs{\cP})$.  Now carrying out this procedure for each part of $\cP$
gives an equipartition $\cQ'$ with the desired property.

Let $\cR$ be the common refinement of $\cQ$ and $\cQ'$. Because the stepping
operator is contractive with respect to the cut norm (i.e.,
$\norm{U_\cR}_\square \leq \norm{U}_\square$),
\begin{align*}
  \norm{W - W_{\cQ'}}_\square
  &\leq
  \norm{W - W_{\cQ}}_\square + \norm{W_\cQ - W_\cR}_\square +
  \norm{W_\cR - W_{\cQ'}}_\square
  \\
  &=
  \norm{W - W_{\cQ}}_\square + \norm{(W_\cQ - W)_\cR}_\square + \norm{W_\cR - W_{\cQ'}}_\square
  \\
  &\leq
    2\norm{W - W_{\cQ}}_\square + \norm{W_\cR - W_{\cQ'}}_\square.
\end{align*}
Thus, it will suffice to bound $\norm{W_\cR - W_{\cQ'}}_\square$ by $2
\norm{W}_p \paren{2\abs{\cQ}/(k\abs{\cP})}^{1-1/p}$.

Let $S$ be the union of the parts of $\cQ'$ that were broken up in its
refinement $\cR$. These are exactly the parts that intersect more than one
part of $\cQ$, so $\lambda(S) \leq \abs{\cQ}/(k\abs{\cP})$. Using the
agreement of $W_{\cQ'}$ with $W_\cR$ on $S^c \x S^c$ (where $S^c := [0,1]
\setminus S$), H\"older's inequality with $1/p+1/p'=1$, the bound
$\snorm{W_\cR}_p \leq \snorm{W_{\cQ'}}_p \leq \snorm{W}_p$, and the triangle
inequality, we get
\begin{align*}
  \norm{W_\cR - W_{\cQ'}}_\square &\leq \norm{W_\cR - W_{\cQ'}}_1\\
    &=\norm{(W_\cR - W_{\cQ'})(1 - 1_{S^c \x S^c})}_1 \\
    &\leq \norm{W_\cR - W_{\cQ'}}_p \norm{1 - 1_{S^c \x S^c}}_{p'}\\
    &= \norm{W_\cR - W_{\cQ'}}_p \big(2 \lambda(S) - \lambda(S)^2\big)^{1-1/p}\\
    &\leq 2\norm{W}_p \big(2 \lambda(S) \big)^{1-1/p}\\
    &\leq 2 \norm{W}_p \paren{\frac{2\abs{\cQ}}{k\abs{\cP}}}^{1-1/p},
\end{align*}
as desired.
\end{proof}

The following lemma is the $L^2$ version of Corollary~3.4(i) in
\cite{BCLSV1}, which in fact never required the $L^\infty$ hypothesis
implicitly assumed there.

\begin{lemma}[Weak regularity lemma for $L^2$ graphons, equitable version] \label{lem:W-L^2-reg-equi}
Let $0 < \e < 1/3$ and let $W \colon [0,1]^2 \to \RR$ be an $L^2$ graphon.
Let $\cP$ be an equipartition of $[0,1]$. Then for every integer $k \geq
4^{10/\e^2}$ there exists an equipartition $\cQ$ refining $\cP$ into exactly
$k \abs{\cP}$ parts so that
\begin{equation} \label{eq:L^2-reg-equi}
\norm{W - W_\cQ}_\square \leq \e \norm{W}_2.
\end{equation}
\end{lemma}

\begin{proof}
  Apply Lemma~\ref{lem:W-L^2-reg} to obtain a refinement
  $\cQ$ of $\cP$ into at most $4^{9 / \e^2} \abs{\cP}$
  parts so that $\norm{W - W_\cQ} \leq \frac13 \e \norm{W}_2$. Now apply
  Lemma~\ref{lem:equitize} with $p=2$ to obtain a refinement $\cQ'$ of
  $\cP$ into an equipartition of exactly $k
  \abs{\cP}$  parts satisfying
  \[
  \norm{W - W_{\cQ'}}_\square
  \leq 2\norm{W - W_{\cQ}}_\square + 2 \norm{W}_2
  \sqrt{\frac{2\abs{\cQ}}{k
     \abs{\cP}}}
 \leq 2\cdot \frac{\e}{3}\norm{W}_2 + 2 \norm{W}_2 \cdot \frac{\e}{6}
 \leq \e \norm{W}_2.
  \]
  Here we used $\abs{\cQ}/\abs{\cP} \leq 4^{9/\e^2} \leq
  \frac{\e^2}{72} 4^{10/\e^2} \leq \frac{1}{2}(\frac{\e}{6})^2 k$, which holds for
  $0 < \e < 1/3$.
  So $\cQ'$ is the desired partition.
\end{proof}

Lemma~\ref{lem:W-L^2-reg-equi} also works for $L^p$ graphons for all $p \geq
2$ by nesting of norms, as \eqref{eq:L^2-reg-equi} implies $\norm{W -
W_\cQ}_\square \leq \e \norm{W}_p$. Now we deal with the case $1 < p < 2$.

\begin{lemma}[Weak regularity lemma for $L^p$ graphons] \label{lem:W-L^p-reg-equi} Let $1 < p < 2$ and $0 < \e < 1$.  Let $W \colon [0,1]^2 \to
  \RR$ be an $L^p$ graphon. Let $\cP$ be an equipartition of
  $[0,1]$. Then for any integer $k \geq 4^{10 (3/\e)^{p/(p-1)}}$ there
  exists an equipartition $\cQ$ refining
  $\cP$ into exactly $k \abs{\cP}$ parts so that
  \[
  \norm{W - W_\cQ}_\square \leq
  \e\norm{W}_p.
  \]
\end{lemma}

Note that as $p \nearrow 2$, the exponent $p/(p-1)$ of $1/\e$ in $k$ in the
lemma tends to $2$, which is the best possible exponent in the bound for the
weak regularity lemma when $p \geq 2$ by \cite{CF}.

\begin{proof}
  Set $K = (3/\e)^{1/(p-1)} \norm{W}_p$, and let
  \[
    W' = W1_{\abs{W} \leq K}.
  \]
  We have
  \begin{align*}
  \norm{W'}_2 &= \norm{W 1_{\abs{W} \leq K}}_2\\
  &\leq \norm{W (K/\abs{W})^{1-p/2}}_2\\
  &= \norm{W}_p^{p/2} K^{1-p/2}
  = (3/\e)^{\frac{2-p}{2(p-1)}} \norm{W}_p.
  \end{align*}
  By Lemma~\ref{lem:W-L^2-reg-equi} there exists an equitable
  partition $\cQ$ refining $\cP$ into exactly $k\abs{\cP}$ parts so that
  \[
  \norm{W' - W'_\cQ}_\square \leq \paren{\frac{\e}{3}}^{\frac{p}{2(p-1)}} \norm{W'}_2 \leq \frac{\e}{3} \norm{W}_p.
  \]
  We also have
  \begin{align*}
  \norm{W_\cQ - W'_\cQ}_1
  &= \norm{(W - W')_\cQ}_1\\
  &\leq \norm{W - W'}_1
  = \norm{W 1_{\abs{W} > K}}_1 \\
  &\leq \norm{W (\abs{W}/K)^{p-1}}_1
  = \norm{W}_p^p / K^{p-1} = \frac{\e}{3} \norm{W}_p.
\end{align*}
It follows that
\begin{align*}
\norm{W - W_{\cQ}}_\square
&\leq
\norm{W - W'}_\square + \norm{W' - W'_{\cQ}}_\square + \norm{W'_{\cQ} - W_{\cQ}}_\square
\\
&\leq
\norm{W - W'}_1 + \norm{W' - W'_{\cQ}}_\square + \norm{W'_{\cQ} - W_{\cQ}}_1
\\
&\leq \frac{\e}{3}\norm{W}_p + \frac{\e}{3}\norm{W}_p +
\frac{\e}{3}\norm{W}_p = \e \norm{W}_p.
\end{align*}
Therefore $\cQ$ is the desired partition.
\end{proof}

Now we prove that the $L^p$ ball is compact with respect to the cut metric.

\begin{proof}[Proof of Theorem~\ref{thm:compact}]

The proof of the theorem is a small modification of the argument in
\cite[Theorem~5.1]{LS07}, with adaptations to the $L^p$ setting.  We begin by
using the weak regularity lemmas to produce approximations to the sequence
$(W_n)_{n \ge 0}$.  The approximations using a fixed number of parts are
easier to analyze than the original sequence, because they involve only a
finite amount of information. We take limits of these approximations and show
that they form a martingale as one varies the number of parts. The limit of
the original sequence is then derived using the martingale convergence
theorem.

By scaling we may assume without loss of generality that $C = 1$. For each
$k$ and $n$ we construct an equipartition $\cP_{n,k}$ using
Lemma~\ref{lem:W-L^2-reg-equi} (when $p \geq 2$) or
Lemma~\ref{lem:W-L^p-reg-equi} (when $1 < p < 2$), so that
\[
\norm{W_n - (W_n)_{\cP_{n,k}}}_\square \leq 1/k.
\]
In doing so, we may assume that $\cP_{n,k+1}$ always refines $\cP_{n,k}$ and
that $|\cP_{n,k}|$ is independent of $n$.

The first step is to change variables so the partitions $\cP_{n,k}$ become
the same.  Let $\cP_k$ be a partition of $[0,1]$ into $|\cP_{n,k}|$ intervals
of equal length, and for each $n$ and $k$, let $\sigma_{n,k}$ be a
measure-preserving bijection from $[0,1]$ to itself that transforms
$\cP_{n,k}$ into $\cP_k$. (This can always be done; see, for example,
Theorem~A.7 in \cite{Jan}.)  Now let
\[
W_{n,k} = \big(W_n^{\sigma_{n,k}}\big)_{\cP_k} =
\big((W_n)_{\cP_{n,k}}\big)^{\sigma_{n,k}}.
\]
Then $W_{n,k}$ is a step-function with interval steps formed from $\cP_k$,
and
\[
\delta_\square(W_n,W_{n,k}) \le 1/k.
\]

Since each interval of $\cP_{k}$ has length exactly $1/|\cP_k|$ and the
stepping operator is contractive with respect to the $p$-norm,
\[
|\cP_k|^{-2} \norm{W_{n,k}}_\infty^p \leq \norm{W_{n,k}}_p^p \leq \norm{W_n}_p^p
\leq 1.
\]
Thus $\norm{W_{n,k}}_\infty \leq |\cP_k|^{2/p}$.

We next pass to a subsequence of $(W_n)_{n \ge 0}$ such that for each $k$,
$W_{n,k}$ converges to a limit $U_k$ almost everywhere as $n \to \infty$. For
each fixed $k$, this is easily done using compactness of a
$|\cP_k|^2$-dimensional cube, because the function $W_{n,k}$ is determined by
$|\cP_k|^2$ values corresponding to pairs of parts in $\cP_k$ and
$\norm{W_{n,k}}_\infty$ is uniformly bounded. To find a single subsequence
that ensures convergence for all $k$, we iteratively choose a subsequence for
$k=1,2,\dots$.

For each $k$, the limit $U_k$ is a step function with $|\cP_k|$ steps such
that $\norm{W_{n,k} - U_k}_p \to 0$ as $n \to \infty$.  In particular, this
implies that $\norm{U_k}_p \leq 1$ for all $k$, since $\norm{W_{n,k}}_p \leq
\norm{W_n}_p \leq 1$ for all $n$ and $k$.

The crucial property of the sequence $U_1,U_2,\dots$ is that it forms a
martingale on $[0,1]^2$ with respect to the $\sigma$-algebras generated by
the products of the parts of $\cP_1,\cP_2,\dots$.  In other words,
$(U_{k+1})_{\cP_k} = U_k$. This follows immediately from
\[
(W_{n,k+1})_{\cP_k} = \big(W_n^{\sigma_{n,k+1}}\big)_{\cP_k}
= \big((W_n)_{\cP_{n,k}}\big)^{\sigma_{n,k+1}} =  W_{n,k}.
\]
(Note that $\sigma_{n,k+1}$ transforms $\cP_{n,k}$ into $\cP_k$ because it
does the same for their refinements $\cP_{n,k+1}$ and $\cP_{k+1}$.)

By the $L^p$ martingale convergence theorem \cite[Theorem~5.4.5]{Dur}, there
exists some $W \in L^p([0,1]^2)$ such that $\norm{U_k - W}_p \to 0$ as $k \to
\infty$. Since $\norm{U_k}_p \leq 1$ for all $k$, we have $\norm{W}_p \leq
1$.

Now $W$ is the desired limit, because
\begin{align*}
\delta_\square(W_n, W)
&\leq \delta_\square(W_n,W_{n,k}) + \delta_\square(W_{n,k},U_k) + \delta_\square(U_k,W) \\
&\leq \delta_\square(W_n,W_{n,k}) + \norm{W_{n,k}-U_k}_1 + \norm{U_k-W}_1.
\end{align*}
Each of the terms in this bound can be made arbitrarily small by choosing $k$
and then $n$ large enough.  Thus, $\delta_\square(W_n, W) \to 0$ as $n \to
\infty$, as desired (keeping in mind that we have passed to a subsequence).
\end{proof}

\section{Regularity lemma for $L^p$ upper regular graph(on)s} \label{sec:reg-lem-up-reg}

In this section we prove a regularity lemma for $L^p$ upper regular graphs
and graphons. This forms the ($L^p$ upper regular sequence) $\to$ ($L^p$
graphon sequence) arrow in Figure~\ref{fig:summary}. We will first present
the proof for graphons, since the notation is somewhat simpler. Then we will
explain the minor modifications needed to prove the result for weighted
graphs. The difference between the two settings is that for graphs, the
partitions of $[0,1]$ in the corresponding graphon need to respect the
atomicity of the vertices, but this is only a minor inconvenience since the
$L^p$ upper regularity condition ensures that no vertex has weight too large.

The main ideas of the proof are as follows. Suppose $W$ is a
$(C,\eta)$-upper $L^p$ regular graphon with $p \ge 2$.  We would like to
proceed as in the proof of the $L^2$ weak regularity lemma, by constructing
partitions $\cP_0,\cP_1,\dots$ such that if $\norm{W-W_{\cP_i}}_\square >
C\e$, then
\[
\norm{W_{\cP_{i+1}}}_2^2 \geq \norm{W_{\cP_{i}}}_2^2 +(C\e)^2.
\]
Furthermore, we would like all the parts of $\cP_i$ to have measure at least
$\eta$, so that $\norm{W_{\cP_i}}_2 \le \norm{W_{\cP_i}}_p \le C$.  These
bounds cannot both hold for all $i$, so we must eventually have
$\norm{W-W_{\cP_i}}_\square \le C\e$ for some $i$.

When we try to do this, we run into two problems:
\begin{enumerate}
\item While $\norm{W-W_{\cP_i}}_\square > C\e$ gives sets $S$ and $T$ such
    that $\abs{\ang{W - W_{\cP_i},1_{S \x T}}} > C\e$, the partition
    generated by $\cP_i$, $S$, and $T$ may have a part of size less than
    $\eta$. In that case, we cannot use the upper regularity assumption as
    we proceed.

\item When $p<2$, the $L^2$ increment argument does not work, since we
    only have bounds on $\norm{W_{\cP_i}}_p$, not $\norm{W_{\cP_i}}_2$.
\end{enumerate}

To deal with the first problem, we will modify $S$ and $T$ to $S'$ and $T'$
such that the new partition has large enough parts, while $\abs{\ang{W -
W_{\cP_i},1_{S' \x T'}}} > C\e/2$.  To do so, we will need a technical
lemma, Lemma~\ref{lem:W-stab} below, which allows us to bound the difference
between these inner products, and which itself follows from a simpler lemma,
Lemma~\ref{lem:small-cut}. After stating and proving these lemmas, we will
formulate Theorem~\ref{thm:W-upper-reg}, which is the regularity lemma
version of Proposition~\ref{prop:densify} for graphons.  In its proof, we
deal with the first problem as describe above, while we deal with the second
by a suitable truncation argument.

We begin with a lemma that bounds the weight of $W$ on $1_{S \x T}$ when one
of $S$ and $T$ is small. Recall that $\lambda$ denotes Lebesgue measure.

\begin{lemma} \label{lem:small-cut}
  Assume $\eta < 1/9$. Let $W \colon
  [0,1]^2 \to \RR$ be a $(C,\eta)$-upper $L^p$ regular graphon, and let $S
  , T \subseteq [0,1]$ be measurable subsets. If $\lambda(S) \leq
  \delta$ for some $\delta \geq \eta$, then
\[
\abs{\ang{W, 1_{S \x T}}} \leq 10 C \delta^{1-1/p}.
\]
\end{lemma}

\begin{proof}
  We prove the lemma in three steps.

  \emph{Step 1.} Let $\cP$ be the smallest partition of $[0,1]$ that
  simultaneously refines $S$ and $T$ (i.e., the parts are $S \cap T,
  S^c \cap T, S \cap T^c, S^c \cap T^c$, excluding empty parts, where $S^c
  := [0,1] \setminus S$). If all
  parts of $\cP$ have measure at least $\eta$, then we can apply
  H\"older's inequality (with $1/p + 1/p' = 1$) and the $(C,\eta)$-upper $L^p$
  regularity
  hypothesis to conclude
  \[
  \abs{\ang{W, 1_{S \x T}}}
  = \abs{\ang{W_\cP, 1_{S \x T}}}
  \leq \norm{W_\cP}_p \norm{1_{S \x T}}_{p'}
  \leq C (\lambda(S)\lambda(T))^{1- 1/p}.
  \]

  \emph{Step 2.} In this step we assume that $3 \eta \leq \lambda(T)
  \leq 1 - 3\eta$. The partition $\cP$ generated by $S$ and $T$ as in
  Step 1 might not satisfy the condition of all parts having measure
  at least $\eta$. Define $S_1 \subseteq T$ and $S_2 \subseteq T^c$ as
  follows.

  If $\lambda(S \cap T) < \eta$, then let $S_1$ be an
  arbitrary
  subset of $T \setminus S$ with $\lambda(S_1) = \eta$; else, if $\lambda(S^c \cap T) < \eta$
  (equivalently, $\lambda(S
  \cap T) > \lambda(T) - \eta$), then let $S_1$ be an arbitrary subset
  of $S \cap T$ with $\lambda(S_1) = \eta$; else, let $S_1 = \emptyset$.

  Similarly, if $\lambda(S \cap T^c) < \eta$, then let $S_2$ be an
  arbitrary
  subset of $T^c \setminus S$ with $\lambda(S_2) = \eta$; else, if $\lambda(S
  \cap T^c) > \lambda(T^c) - \eta$, then let $S_2$ be an arbitrary subset
  of $S \cap T^c$ with $\lambda(S_2) = \eta$; else, let $S_2 =
  \emptyset$.

  Let $S' = S \symmdiff S_1 \symmdiff S_2$ (where $\symmdiff$ denotes the
  symmetric difference, and here each $S_i$ is either
  contained in $S$ or disjoint from $S$). Note that the pairs $(S_1, T)$,
  $(S_2, T)$, $(S', T)$ all satisfy the hypotheses of Step 1. So we
  have
  \begin{align*}
  \abs{\ang{W, 1_{S \x T}}}
  &= \abs{\ang{W, 1_{S' \x T} \pm 1_{S_1\x T} \pm 1_{S_2\x T}}}
  \\
  &\leq \abs{\ang{W, 1_{S' \x T}}} + \abs{\ang{W,1_{S_1 \x T}}} +
  \abs{\ang{W,1_{S_2 \x T}}}
  \\
  &\leq C(\lambda(S')\lambda(T))^{1-1/p} +
  C(\lambda(S_1)\lambda(T))^{1-1/p} +
  C(\lambda(S_2)\lambda(T))^{1-1/p}
  \\
  &\leq C(\lambda(S) + 2\eta)^{1-1/p} + 2C\eta^{1-1/p}
  \\
  &\leq 5C \delta^{1-1/p}.
\end{align*}
The last step follows from the assumption $\lambda(S) \leq \delta$ and $\delta \geq \eta$.

\emph{Step 3.} Now we relax the $3\eta \leq \lambda(T) \leq 1 - 3\eta$
assumption. If $\lambda(T) < 3\eta$, then let $T_1$ be any subset of $T^c$
with $\lambda(T_1) = 3\eta$; else, if $\lambda(T) > 1-3\eta$, then let $T_1$
be any subset of $T$ with $\lambda(T_1) = 3\eta$; else, let $T_1 =
\emptyset$. Let $T' = T \symmdiff T_1$. Then $3\eta \leq \lambda(T') \leq 1 -
3\eta$. So applying Step 2, we have
\[
  \abs{\ang{W, 1_{S \x T}}}
  \leq \abs{\ang{W, 1_{S \x T'}}} +  \abs{\ang{W, 1_{S \x T_1}}}
  \leq 10C\delta^{1-1/p}. \qedhere
\]
\end{proof}

\begin{lemma} \label{lem:W-stab}
Assume $\eta < 1/9$. Let $W \colon [0,1]^2 \to \RR$ be a $(C,\eta)$-upper
$L^p$ regular graphon. Let $S, S', T, T' \subseteq [0,1]$ be measurable sets
satisfying $\lambda(S \symmdiff S'), \lambda(T \symmdiff T') \leq \delta$,
for some $\delta \geq \eta$. Then
\[
\abs{\ang{W, 1_{S \x T} - 1_{S' \x T'}}} \leq 40 C \delta^{1-1/p}.
\]
\end{lemma}

\begin{proof}
  We have
  \[
  1_{S \x T} - 1_{S' \x T'}
  = 1_{(S \setminus S') \x T} + 1_{(S \cap S') \x (T \setminus T')}
    - 1_{(S' \setminus S) \x T'} - 1_{(S \cap S') \x (T' \setminus T)}.
  \]
  Applying Lemma~\ref{lem:small-cut} to each of the four terms below and
  using $\lambda(S\setminus S'), \lambda(S' \setminus S),
  \lambda(T \setminus T'), \lambda(T'\setminus T) \leq \delta$,
  we have
  \begin{align*}
    \abs{\ang{W, 1_{S \x T} - 1_{S' \x T'}}}
     &\le \abs{\ang{W,1_{(S \setminus S') \x T}}}
        + \abs{\ang{W, 1_{(S \cap S') \x (T \setminus T')}}}\\
        & \quad \phantom{} + \abs{\ang{W, 1_{(S' \setminus S) \x T'}}} 
        + \abs{\ang{W, 1_{(S \cap S') \x (T' \setminus T)}}}
    \\
    &\leq 4 \cdot 10C \delta^{1-1/p}. \qedhere
  \end{align*}
\end{proof}

\begin{theorem}[Weak regularity lemma for $L^p$ upper regular graphons] \label{thm:W-upper-reg}
Let $C >0$, $p > 1$, and $0 < \e < 1$. Set
$N =  (6/\e)^{\max\{2,p/(p-1)\}}$ and
$\eta = 4^{-N-1}(\e/160)^{p/(p-1)}$.
Let $W \colon [0,1]^2 \to \RR$ be a $(C,\eta)$-upper $L^p$
regular graphon. Then there exists a partition
$\cP$ of $[0,1]$ into at most $4^N$ measurable parts, each having
measure at least $\eta$, so that
\[
\norm{W - W_\cP}_\square \leq C \e.
\]
\end{theorem}

Proposition~\ref{prop:densify} for graphons follows as an immediate
corollary.

\begin{proof}
We consider a sequence of partitions $\cP_0, \cP_1, \cP_2, \dots,
\cP_n$ of $[0,1]$, starting with the trivial partition $\cP_0 =
\{[0,1]\}$. The following properties will be maintained:
\begin{enumerate}
  \item The partition $\cP_{i+1}$ refines $\cP_i$ by dividing each part
      of $\cP_i$ into at most four subparts. So in particular
      $\abs{\cP_i} \leq 4^i$.
  \item For each $i$, all parts of $\cP_i$ have measure at least
    $\eta$.
\end{enumerate}
These partitions are constructed as follows. For each $0 \leq i < n$, if $\cP_i$
satisfies
$\norm{W - W_{\cP_i}}_\square \leq C\e$, then we have found the desired
partition. Otherwise, there exists measurable subsets
$S,T \subseteq [0,1]$ with
\begin{equation} \label{eq:W-cut-dev}
\abs{\ang{W - W_{\cP_i},1_{S \x T}}} > C\e.
\end{equation}
Next we find $S', T'$ so that $\lambda(S\symmdiff S'), \lambda(T \symmdiff
T') \leq 2\abs{\cP_i}\eta$, such that if we define $\cP_{i+1}$ to be the
common refinement of $\cP$, $S'$, and $T'$, then all parts of $\cP_i$ have
size at least $\eta$. Indeed, look at the intersection of $S$ with each part
of $\cP_i$, and obtain $S'$ from $S$ by deleting (rounding down) the parts
that intersect with $S$ in measure less than $\eta$, and then adding
(rounding up) the parts that intersect $S^c$ in measure less than $\eta$.
Let $\cP_{i+1/2}$ be the common refinement of $\cP_i$ and $S'$, so that all
parts of $\cP_{i+1/2}$ have measure at least $\eta$, and $\lambda(S\symmdiff
S') \leq \abs{\cP_i}\eta$. Next, do a similar procedure to $T$ to obtain
$T'$ so that the common refinement $\cP_{i+1}$ of $\cP_{i+1/2}$ and $T'$ has
all parts with measure at least $\eta$. Here we have $\lambda(T \symmdiff
T') \leq \abs{\cP_{i+1/2}} \eta \leq 2\abs{\cP_i} \eta$. So $\cP_{i+1}$ has
the desired properties.

If the construction of the sequence $\cP_0, \dots, \cP_n$ of partitions stops
with $n \leq N$, then we are done. Otherwise let us stop the sequence at
$\cP_n$ with $n = \ceil{N}$. We will derive a contradiction.

Let $0 \leq i < n$, and let $S, S', T, T'$ be the sets used to construct
$\cP_{i+1}$ from $\cP_i$. Using $\lambda(S\symmdiff S'), \lambda(T \symmdiff
T') \leq 2\abs{\cP_i}\eta \leq 2 \cdot 4^N\eta$, we have by
Lemma~\ref{lem:W-stab}
\begin{equation} \label{eq:W-alter1}
\abs{\ang{W,1_{S \x T} - 1_{S' \x T'}}} \leq 40C(2 \cdot 4^N\eta)^{1-1/p}
\leq  C\e/4.
\end{equation}
Also by H\"older's inequality (with $1/p + 1/p' = 1$),
\begin{equation} \label{eq:W-alter2}
\begin{split}
	\abs{\ang{W_{\cP_i},1_{S \x T} - 1_{S' \x T'}}}
	&\leq \norm{W_{\cP_i}}_p \norm{1_{S \x T} - 1_{S' \x T'}}_{p'} \\
	&\leq C (\lambda(S \symmdiff S') + \lambda(T \symmdiff T'))^{1/p'}\\
	& \leq C (4 \cdot 4^N\eta)^{1-1/p} \leq C\e/160
	\leq C\e/8.
\end{split}
\end{equation}
It follows that
\begin{align*}
\abs{\ang{W - W_{\cP_i},1_{S \x T}} - \ang{W - W_{\cP_i},1_{S' \x
			T'}}}
&\leq \abs{\ang{W,1_{S \x T} - 1_{S' \x T'}}}\\
&\quad \phantom{} + \abs{\ang{W_{\cP_i},1_{S \x T} - 1_{S' \x
			T'}}}\\ 
&\leq C\e/2.
\end{align*}
Combing the above inequality with \eqref{eq:W-cut-dev} gives us
\[
\abs{\ang{W - W_{\cP_i},1_{S' \x T'}}} > C\e/2.
\]
Since $S'$ and $T'$ are both unions of parts in $\cP_{i+1}$, we have $\ang{W ,1_{S' \x
		T'}} = \ang{W_{\cP_{i+1}},1_{S' \x T'}}$, so
\begin{equation} \label{eq:W-cut-dev'}
\abs{\ang{W_{\cP_{i+1}} - W_{\cP_i},1_{S' \x T'}}} > C\e/2.
\end{equation}

We consider two cases: $p\geq 2$ and $ 1 < p < 2$.

\medskip 

\noindent\emph{Case I: $p \geq 2$.} This case is easier. Since $\cP_{i+1}$ is a
refinement of $\cP_i$, we have $\ang{W_{\cP_{i+1}} - W_{\cP_i}, W_{\cP_i}} = 0$. So by
the Pythagorean theorem, followed by the Cauchy-Schwarz inequality,
\[
\norm{W_{\cP_{i+1}}}_2^2 - \norm{W_{\cP_{i}}}_2^2
= \norm{W_{\cP_{i+1}} - W_{\cP_{i}}}_2^2
\geq \abs{\ang{W - W_{\cP_i},1_{S' \x T'}}}^2
> C^2\e^2/4.
\]
So $\norm{W_{\cP_n}}_2^2 > n C^2\e^2/4 \geq N C^2\e^2/4
> C^2$, which contradicts
$\norm{W_{\cP_n}}_2 \leq \norm{W_{\cP_n}}_p \leq C$.

\medskip 

\noindent\emph{Case II: $1 < p < 2$.} In this case, we no longer have an
upper bound on $\norm{W_{\cP_n}}_2$ as before. We proceed by truncation: we
stop the partition refinement process at step $n$, truncate the last step
function, and then look back to calculate the energy increment that would
have come from doing the same partition refinement on the truncated graphon.
Set
\[
K := C(6/\e)^{1/(p-1)},
\]
and define the truncation
\[
U := W_{\cP_n} 1_{\abs{W_{\cP_n}} \leq K}.
\]
We claim that for $0 \leq i < n$,
\begin{equation} \label{eq:U-L2-inc}
  \norm{U_{\cP_{i+1}}}_2^2 > \norm{U_{\cP_i}}_2^2 + (C\e/6)^2.
\end{equation}
Then one has $\norm{U_{\cP_n}}_2^2 > n(C\e/6)^2 \geq N (C\e/6)^2 = C^2
(6/\e)^{(2-p)/(p-1)}$, which contradicts
\begin{align*}
\norm{U_{\cP_n}}_2^2
= \norm{W_{\cP_n} 1_{\abs{W_{\cP_n}} \leq K}}_2^2
&\leq \norm{W_{\cP_n} (K/\abs{W_{\cP_n}})^{1-p/2}}_2^2 \\
&= \norm{W_{\cP_n}}_p^p K^{2 - p} \leq C^pK^{2-p} = C^2 (6/\e)^{(2-p)/(p-1)}.
\end{align*}
It remains to prove \eqref{eq:U-L2-inc}. We have
\begin{align*}
\norm{W_{\cP_n} - U}_1
&= \snorm{W_{\cP_n} 1_{\abs{W_{\cP_n}} > K}}_1\\
&\leq \snorm{W_{\cP_n} (\abs{W_{\cP_n}} /K)^{p-1}}_1 \\
&=  \norm{\abs{W_{\cP_n}}^p}_1 / K^{p-1}
= \snorm{W_{\cP_n}}_p^p / K^{p-1}\\
&\leq C^p/K^{p-1}
= C\e/6.
\end{align*}
Since $\cP_n$ is a refinement of $\cP_i$, we have $(W_{\cP_n})_{\cP_i}
= W_{\cP_i}$. So
\begin{equation}\label{eq:W-U}
\norm{W_{\cP_i} - U_{\cP_i}}_1
=\norm{(W_{\cP_n} - U)_{\cP_i}}_1
\leq \norm{W_{\cP_n} - U}_1
\leq C\e/6.
\end{equation}
Similarly, $\snorm{W_{\cP_{i+1}} - U_{\cP_{i+1}}}_1 \leq C\e/6$. Using the triangle
inequality, \eqref{eq:W-cut-dev'}, and \eqref{eq:W-U}, we find that
\begin{align*}
\abs{\ang{U_{\cP_{i+1}} - U_{\cP_i}, 1_{S'\x T'}}}
&\geq \abs{\ang{W_{\cP_{i+1}} - W_{\cP_i}, 1_{S'\x T'}}}\\
&\quad \phantom{} - \norm{W_{\cP_i} - U_{\cP_i}}_1 - \norm{W_{\cP_{i+1}} - U_{\cP_{i+1}}}_1
\\
&> C(\e/2 - \e/6 - \e/6) = C\e/6.
\end{align*}
Since $\cP_{i+1}$ is a refinement of $\cP_i$, we have $\ang{U_{\cP_{i+1}} - U_{\cP_i}, U_{\cP_i}} = 0$.
So by the Pythagorean theorem, followed by the Cauchy-Schwarz inequality, we have
\[
\norm{U_{\cP_{i+1}}}_2^2 - \norm{U_{\cP_{i}}}_2^2
= \norm{U_{\cP_{i+1}} - U_{\cP_{i}}}_2^2
\geq \abs{\ang{U_{\cP_{i+1}} - U_{\cP_i}, 1_{S'\x T'}}}^2
> (C\e/6)^2,
\]
which proves \eqref{eq:U-L2-inc}, as desired.
\end{proof}

This completes the proof of the weak regularity lemma for $L^p$ upper
regular graphons.

\begin{remark} \label{remark:graphon-equi}
At the cost of slightly worse constants, the statement of
Theorem~\ref{thm:W-upper-reg} can be strengthened to provide an
equipartition. To this end, we first apply the theorem to $W$, obtaining a
partition $\cP_0$ into  at most $4^N$ parts such that each part has size at
least $\eta$ and $\norm{W - W_{\cP_0}}_\square \leq C \e$.  Since $W$ is
assumed to be $L^p$ upper regular, we obtain a graphon $U=W_{\cP_0}$ such
that $\norm{U}_p\leq C$. Depending on whether $p\geq 2$ or $p\in(1,2)$, we
then apply Lemma~\ref{lem:W-L^2-reg-equi} or Lemma~\ref{lem:W-L^p-reg-equi}
to $U$ and the trivial partition of $[0,1]$ consisting of the single class
$[0,1]$. As a consequence, for $k\geq 4^{\max\{10/\e^2,10(3/\e)^{p/(p-1)}\}}$
we can find an equipartition $\cP$ of $[0,1]$ into $k$ parts such that
$\norm{W_{\cP_0} - U_\cP}_\square=\norm{U - U_\cP}_\square \leq C \e$. With
the help of the triangle inequality, this implies
\[
\norm{W - U_\cP}_\square \leq 2C \e.
\]
But $U_\cP$ is a step functions with steps in $\cP$, and it should
approximate $W$ at most as well as $W_\cP$.  While this is not quite true, it
is true at the cost of another factor of two.  To see this, we use the
triangle inequality, $U_\cP=(U_\cP)_\cP$, and the fact that the stepping
operator is a contraction with respect to the cut norm to bound
\[
\begin{aligned}
\norm{W - W_\cP}_\square&\leq
\norm{W - U_\cP}_\square + \norm{W_\cP - U_\cP}_\square\\
&=\norm{W - U_\cP}_\square + \norm{(W - U_\cP)_\cP}_\square
\\
&\leq
\norm{W - U_\cP}_\square + \norm{W - U_\cP}_\square\\
& = 2 \norm{W-U_\cP}_\square.
\end{aligned}
\]
Putting everything together, we see that for any $k\geq
4^{\max\{10/\e^2,10(3/\e)^{p/(p-1)}\}}$ we can find an equipartition $\cP$ of
$[0,1]$ into exactly $k$ parts such that
\[
\norm{W - W_\cP}_\square\leq 4C\e,
\]
provided $W$ is $(C,\eta)$-upper $L^p$ regular with  $\eta =
4^{-N-1}(\e/160)^{p/(p-1)}$, where $N =  (6/\e)^{\max\{2,p/(p-1)\}}$.
\end{remark}

Next we state the analogue of Theorem~\ref{thm:W-upper-reg} for weighted
graphs and explain how to modify the above proof to work for weighted
graphs.

If $G$ is a weighted graph, and $\cP = \{V_1, \dots, V_m\}$ is a partition of
$V(G)$, then we denote by $G_\cP$ the weighted graph on $V(G)$ (with the same
vertex weights as $G$) and edge weights as follows. For $s \in V_i, t \in
V_j$ the edge between $s$ and $t$ is given weight
\[
\beta_{st}(G_\cP) = \sum_{x \in V_i, y \in V_j} \frac{\alpha_x
  \alpha_y}{\alpha_{V_i}\alpha_{V_j}} \beta_{xy}(G)
\]
(note that we allow $x = y$). In other words, $G_\cP$ is obtained from $G$ by
averaging the edge weights inside each $V_i \x V_j$. In terms of graphons, we
have $W^{G_\cP} = (W^G)_\cP$, where we abuse notation by letting $\cP$ also
denote the partition of $[0,1]$ corresponding to the vertex partition.

\begin{theorem}[Weak regularity lemma for $L^p$ upper regular graphs] \label{thm:G-upper-reg-lemma} Let $C > 0$, $p > 1$, and $0
  < \e < 1$. Set $N = (6/\e)^{\max\{2,p/(p-1)\}}$ and $\eta = 4^{-N-1}
  (\e/320)^{p/(p-1)}$. Let $G = (V,E)$ be a $(C,\eta)$-upper $L^p$
  regular weighted graph. Then there exists a partition
  $\cP$ of $V$ into at most $4^N$ parts, each having weight at least $\eta
  \alpha_G$, so that
  \[
  d_\square\paren{\frac{G}{\norm{G}_1}, \frac{G_\cP}{\norm{G}_1}} \leq C \e.
  \]
\end{theorem}

Let us explain how one can modify the proofs in this section to prove
Theorem~\ref{thm:G-upper-reg-lemma}. The only difference is that in the
proceeding proofs, instead of taking arbitrary measurable sets, we are only
allowed to take subsets of $[0,1]$ corresponding to subsets of vertices. Another
way to view this is that we are working with a different $\sigma$-algebra on
$[0,1]$, where the new $\sigma$-algebra comes from a partition of $[0,1]$
into parts with measure equal to the vertex weights (as a fraction of the
total vertex weights) of $G$. So previously in certain steps of the argument
in Lemma~\ref{lem:small-cut} where we took an arbitrary subset $S_1$ a
certain specified measure (say $\lambda(S_1) = \eta$), we have to be content
with just having $\lambda(S_1) \in [\eta, 2\eta)$. This can be done since the
$(C,\eta)$-upper $L^p$  regularity assumption implies no vertex occupies
measure greater than $\eta$ times the total vertex weight.

With this modification in place, Lemma~\ref{lem:small-cut} then
becomes the following.

\begin{lemma}
  Assume $\eta < 1/13$. Let $G$ be a
  $(C,\eta)$-upper $L^p$ regular weighted graph with vertex weights
  $\alpha_i$ and edge weights $\beta_{ij}$. Let $S,
  T \subseteq V(G)$. If $\alpha_S \leq \delta \alpha_G$ for some $\delta
  \geq \eta$, then
  \[
  \abs{\sum_{s \in S, t \in T} \beta_{st}} \leq 20 \delta^{1-1/p}
  \sum_{i,j\in V(G)} \abs{\beta_{ij}}.
  \]
\end{lemma}

The conclusion of Lemma~\ref{lem:W-stab} must be changed similarly, with the
bound increased by a factor of 2. To prove Theorem~\ref{thm:G-upper-reg-lemma}
we can modify the proof of Theorem~\ref{thm:W-upper-reg} to allow only
subsets of vertices instead of arbitrary measurable sets.

\begin{remark} \label{remark:G-upper-reg-equi}
As in Remark~\ref{remark:graphon-equi}, we can achieve an equipartition in
Theorem~\ref{thm:G-upper-reg-lemma} at the cost of worse constants.  Of course
the indivisibility of vertices means we cannot always achieve an exact
equipartition.  Instead, by an \emph{equipartition} of a graph $G$ we mean a
partition of $V(G)$ into $k$ parts $P_1,\dots,P_k$ such that for each $i$,
\[
\abs{\alpha_{P_i} - \frac{\alpha_G}{k}} < \max_{j \in V(G)} \alpha_j.
\]
The argument is the same as in Remark~\ref{remark:graphon-equi}, except that
we must use an equitable weak $L^p$ regularity lemma for graphs, while
Lemmas~\ref{lem:W-L^2-reg-equi} and~\ref{lem:W-L^p-reg-equi} were stated for
graphons.  For $p\ge2$, Corollary~3.4(ii) in \cite{BCLSV1} supplies what we
need, and exactly the same truncation argument used to derive
Lemma~\ref{lem:W-L^p-reg-equi} from Lemma~\ref{lem:W-L^2-reg-equi} extends
this argument to $p<2$.  The only difference is that the bound on $\eta$ is
now inherited from Theorem~\ref{thm:G-upper-reg-lemma} instead of
Theorem~\ref{thm:W-upper-reg}. We conclude that for $k\geq
4^{\max\{10/\e^2,10(3/\e)^{p/(p-1)}\}}$, we can find an equipartition $\cP$
of $V(G)$ into exactly $k$ parts such that
\[
d_{\square}\paren{\frac{G}{\norm{G}_1},\frac{G_\cP}{\norm{G}_1}}\leq 4C\e,
\]
provided $G$ is $(C,\eta)$-upper $L^p$ regular with  $\eta =
4^{-N-1}(\e/320)^{p/(p-1)}$, where $N = (6/\e)^{\max\{2,p/(p-1)\}}$.
\end{remark}

\section{Limit of an $L^p$ upper regular sequence} \label{sec:limit}

Putting together the results in the last two sections, we obtain the
limit for an $L^p$ upper regular sequence, thereby completing the
($L^p$ upper regular sequence) $\to$ ($L^p$ graphon limit) arrow in Figure~\ref{fig:summary}.

\begin{proof}[Proof of
  Theorems~\ref{thm:G-limit} and \ref{thm:W-upper-reg-limit}]
  We give the proof of Theorem~\ref{thm:W-upper-reg-limit} (for
  graphons). The proof of Theorem~\ref{thm:G-limit} (for weighted graphs) is nearly
  identical (using Theorem~\ref{thm:G-upper-reg-lemma} instead of
  Theorem~\ref{thm:W-upper-reg}).

  Let $W_n$ be a upper $L^p$ regular sequence of graphons. In other words, there
  exists a sequence
  $\eta_n \to 0$ so that $W_n$ is $(C+\eta_n, \eta_n)$-upper $L^p$ regular.  Applying
  Theorem~\ref{thm:W-upper-reg}, we can find a sequence $\e_n \to 0$
  so that for each $n$, there exists a partition $\cP_n$ of $[0,1]$ for which
  each part has measure at least $\eta_n$ and
  $\norm{W_n - (W_n)_{\cP_n}}_\square \leq \e_n$. We have
  $\norm{(W_n)_{\cP_n}}_p \leq C + \eta_n$ due to $L^p$ upper
  regularity. By Theorem~\ref{thm:compact}, there exists an $L^p$
  graphon $W$ so that $\norm{W}_p\leq C$ and $\delta_\square((W_n)_{\cP_n}, W) \to 0$ along
  some subsequence. Since $\e_n \to 0$,
  $\delta_\square(W_n,W) \to 0$ along this subsequence.
\end{proof}

The converse, Proposition~\ref{prop:conv-to-Lp}, follows as an
  corollary of the following lemma. (Note that an $L^p$ graphon $W$ is
  automatically $(\norm{W}_p, \eta)$-upper $L^p$ regular for every
  $\eta \geq 0$.)

\begin{lemma} \label{lem:small-cut-upper-regular}
Let $C > 0$, $\eta >0$, and $1 \leq p \leq \infty$, and let $W \colon
[0,1]^2 \to \RR$ be a $(C,\eta)$-upper $L^p$ regular graphon. Let $U \colon
[0,1]^2 \to \RR$ be another graphon. If $\norm{W - U}_\square \leq \eta^3$,
then $U$ is $(C+\eta,\eta)$-upper $L^p$ regular.
\end{lemma}

\begin{proof}
For any subsets $S, T \subseteq [0,1]$, we have $\abs{\ang{W - U, 1_{S\x
T}}} \leq \norm{W - U}_\square \leq \eta^3$. It follows that
\[
\abs{ \frac{1}{\lambda(S)\lambda(T)}\paren{\int_{S \x T} W \,d\lambda - \int_{S\x T} U \,d\lambda}} \leq \frac{\eta^3}{\lambda(S)\lambda(T)} \leq \eta,
\]
provided $\lambda(S), \lambda(T) \geq \eta$. So for any partition $\cP$ of
$[0,1]$ into sets each having measure at least $\eta$ we have $\abs{U_{\cP}
- W_{\cP}} \leq \eta$ pointwise. Therefore,
\[
\norm{U_\cP}_p \leq \norm{\abs{W_\cP} + \eta}_p \leq \norm{W_\cP}_p + \norm{\eta}_p \leq C + \eta.
\]
It follows that $U$ is $(C + \eta,\eta)$-upper $L^p$ regular.
\end{proof}

Next we prove Proposition~\ref{prop:G-no-limit}, which shows that without the
$L^p$ upper regularity assumption, a sequence of a graphs might not have a
Cauchy subsequence (with respect to $\delta_\square$). Furthermore, even a
Cauchy sequence might not have a limit in the form of a graphon.

\begin{proof}[Proof of Proposition~\ref{prop:G-no-limit}]
(a) For each $n \geq 2$, let $G_n$ be a graph on $n2^n$ vertices consisting
of a single clique on $n$ vertices. Then $\norm{G_n}_1 = 2^{-2n} (n-1)/n$.
Let $W_n = W^{G_n}/\norm{G_n}_1$, where the support of $W_n$ is contained in
$[0,2^{-n}]^2$. We claim that $\delta_\square(W_m, W_n) \geq 1/2$ for any $m
\neq n$. Indeed, for any measure-preserving bijection $\sigma \colon [0,1]
\to [0,1]$,
\begin{align*}
\norm{W_m - W_n^{\sigma}}_\square
&\geq \ang{W_m - W_n^{\sigma}, 1_{[0,2^{-m}]^2}}\\
&\geq 1 -  2^{-2m} \snorm{W_n}_\infty\\
&= 1 - 2^{-2(m-n)} n/(n-1)
\geq 1/2
\end{align*}
for $m>n$.

\medskip 

\noindent (b) Our proof is inspired by a classic example of an $L^1$
martingale that converges almost surely but not in $L^1$: a martingale that
starts at $1$ and then at each step either doubles or becomes zero.  The
analogue of this classic example will be a Cauchy sequence of graphs $G_n$
whose normalized graphons converge to zero pointwise almost everywhere but
not in cut distance.  We will build this sequence inductively so that
$G_{n+1}$ is formed from $G_n$ by replacing every edge of $G_n$ with a
quasi-random bipartite graph.

More precisely, for every $n$, let $\e_n = 4^{-n}$, and fix a simple graph
$H_n$ with $\delta_\square(H_n, 1_{[0,1]^2}/2) \leq \e_n$. Let $G_1$ be the
graph with one edge on two vertices. Set $G_{n+1} := G_n \x H_n$. In other
words, to obtain $G_{n+1}$ from $G_n$, replace every vertex $v$ of $G_n$ by
$k = \abs{V(H_n)}$ copies $v_1, \dots, v_k$. The edges of $G_{n+1}$ consists
of $u_i v_j$ where $uv$ is an edge of $G_n$ and $ij$ is an edge of $H_{n}$.

Now we show that $(G_n)_{n \ge 0}$ is a Cauchy sequence with respect to the
normalized cut metric. First, using the natural overlay between $W^{G_n}$ and
$W^{G_{n+1}}$ (the intervals $I_1, \dots, I_{\abs{V(G_n)}}$ corresponding to
the vertices of $G_n$ are each partitioned into $\abs{H_n}$ parts
corresponding to the vertices of $G_{n+1}$), we see that
\[
\delta_\square\paren{G_{n+1}, \frac{1}{2}G_n} \leq
\norm{W^{G_{n+1}}-\frac{1}{2} W^{G_n}}_\square \leq
\norm{W^{H_n}- \frac{1}{2} 1_{[0,1]^2}}_\square \leq \e_n,
\]
since any $\ang{W^{G_{n+1}} - W^{G_n}/2, 1_{A \x B}}$ is equal to the sum of
the contributions from each of the $\abs{V(G_n)}^2$ cells $I_i \x I_j$, and
the contribution from each cell is bounded by $\norm{W^{H_n}-
1_{[0,1]^2}/2}_\square / \abs{V(G_n)}^2$. Note that $\norm{G_{n+1}}_1 /
\norm{G_n}_1 \in [1/2 - \e_n, 1/2 + \e_n]$. It follows that
\begin{align*}
\delta_\square\paren{\frac{G_{n+1}}{\norm{G_{n+1}}_1},
  \frac{G_n}{\norm{G_n}_1}}
&=
\frac{1}{\norm{G_{n+1}}_1} \delta_\square\paren{G_{n+1},
  \frac{\norm{G_{n+1}}_1}{\norm{G_n}_1} G_n}
\\
&\leq
3^{n+1}\paren{\delta_\square\paren{G_{n+1}, \frac{1}{2} G_n} + \e_n}
\\
&\leq 3^{n+1} \cdot 2\e_n = 6 \cdot (3/4)^n.
\end{align*}
Thus the graphs $G_n / \norm{G_n}_1$ form a Cauchy sequence with respect to
$\delta_\square$.

Next we show that $G_n / \norm{G_n}_1$ does not converge to any graphon with
respect to $\delta_\square$. Let $W_n = W^{G_n}/\norm{G_n}_1$ (properly
aligned, so that the support of $W_{n+1}$ is contained in the support of
$W_n$). Then $W_n$ converges to zero pointwise almost everywhere, but zero
cannot be the $\delta_\square$-limit of the sequence since $\EE W_n = 1$ for
all $n$.  Indeed, as we will see shortly, there can be no $U$ such that
$\delta_\square(W_n, U) \to 0$.   Assume by contradiction that there is such
a graphon.  Since $W_n$ is non-negative, $\ang{U, 1_{A \x B}} \geq 0$ for
every $A, B\subseteq [0,1]$, implying that $U$ is nonnegative as well.
Furthermore $\EE U=1$, since $\EE W_n = 1$ and $\abs{\EE W_n-\EE U} \leq
\delta_\square(U,W_n)$  (note that $\EE U=\EE U^\sigma$ for every
measure-preserving bijection $\sigma$). We will show that $U$ has the
following property: for every $\e > 0$, there exists a subset $S \subseteq
[0,1]^2$ with $\lambda(S) \geq 1- \e$ and $\ang{U, 1_S} \leq \e$. It would
then follow that $U \equiv 0$, which is a contradiction.

Now it remains to verify the claim. There exists a
sequence of measure-preserving bijections $\sigma_n \colon [0,1] \to
[0,1]$ such that $\norm{W_n - U^{\sigma_n}}_\square \to 0$. Fix an $m$ with
$\norm{G_m}_1 \leq \e$, and let $S$ be the complement of the support
of $W_m$. So $S$ is the disjoint union of at most
$\abs{V(G_m)}^2$ rectangles and $\lambda(S) \geq 1 - \e$. Choose an $n > m$ so that $\delta_\square(W_n,
U) < \abs{V(G_m)}^{-2} \e$. Since $W_n$ is also zero on $S$, we have $\ang{U^{\sigma_n},
  1_{A \x B}} \leq \delta_\square(W_n,
U) < \abs{V(G_m)}^{-2} \e$ for every rectangle $A \x B$ contained in
$S$. Summing over the at most $\abs{V(G_m)}^2$ such rectangles whose
disjoint union is $S$, we find that $\ang{U^{\sigma_n},1_S}
\leq \e$. The claim then follows.
\end{proof}

The following proposition shows that when dealing with graphs, we can replace
the measure-preserving bijection implicit in $\delta_\square$ with a
permutation of the vertices.

\begin{proposition} \label{prop:order}
Let $C>0$ and $p > 1$, and let $(G_n)_{n \ge 0}$ be a $C$-upper $L^p$ regular
sequence of weighted graphs such that $\delta_\square(G_n/\norm{G_n}_1, U)
\to 0$ for some $L^p$ graphon $U$. Then the vertices of the graphs $G_n$ may
be ordered in such a way that $\norm{W^{G_n}/\norm{G_n}_1 - U}_\square \to
0$.
\end{proposition}

We recall the following lemma\footnote{In \cite{Lov}, Theorem~9.29 is stated
for weighted graphs whose edge weights lie in $[0,1]$, but it immediately
implies the version stated here.} from \cite[Theorem~9.29]{Lov}, where it is
attributed to Alon. Here $\hat\delta_\square(G_1, G_2)$ denotes the cut
distance with respect to the optimal \emph{integral overlay}, i.e.,
$\hat\delta_\square(G_1, G_2) := \min_{G'_1} d_\square(G'_1, G_2)$, where
$G'_1$ is $G_1$ with any reordering of its vertices (assuming $|V(G_1)| =
|V(G_2)|$).

\begin{lemma}\label{lem:integral-vs-fractional-overlay}
For any two weighted graphs $G_1$ and $G_2$ with the same number $v$ of
vertices, unit node weights, and edge weights in $[-1,1]$,
\[
\hat\delta_\square(G_1, G_2) \leq \delta_\square(G_1, G_2) +
\frac{34}{\sqrt{\log v}}.
\]
\end{lemma}

As an immediate corollary, if the graphs in the lemma have edge weights in
$[-K, K]$ instead for some $K > 0$, then the same inequality holds with the
final term replaced by $34K/\sqrt{\log v}$.

Note that it was proved in \cite[Theorem 2.3]{BCLSV1} that
\[
\delta_\square(G_1, G_2) \leq \hat\delta_\square(G_1, G_2) \leq 32
\delta_\square(G_1, G_2)^{1/67}
\]
under the hypotheses of Lemma~\ref{lem:integral-vs-fractional-overlay}. It
remains open whether $\hat\delta_\square(G_1, G_2) = O(
\delta_\square(G_1,G_2))$, which would slightly simplify the proof of
Proposition~\ref{prop:order} if true.

\begin{proof}[Proof of Proposition~\ref{prop:order}]
Let $W_n = W^{G_n}/\norm{G_n}_1$, which depends on the ordering of the
vertices of $G_n$. We need to show that some such ordering of vertices yields
$d_\square(W_n,U) \to 0$, given that $\delta_\square(W_n, U) \to 0$.

First we prove the lemma by a truncation argument under the additional
hypotheses that the graphs
$G_n$ all have unit node weights and $\norm{W_n}_p \leq C$.  We begin by
choosing a sequence of truncations $K_n$ so that $K_n \to \infty$ and $K_n /
\sqrt{\log|V(G_n)|} \to 0$.  (Note that $|V(G_n)| \to \infty$ because
$(G_n)_{n \ge 0}$ is a $C$-upper $L^p$ regular sequence.)

Let $U_n$ denote the step function $U_{\cP_n}$, where $\cP_n$ is the
partition of $[0,1]$ into $|V(G_n)|$ equal length intervals. By
Lemma~\ref{lem:integral-vs-fractional-overlay}, we can reorder the vertices
of $G_n$ so that the corresponding graphon $W_n$ satisfies
\begin{align*}
d_\square\big(W_n 1_{|W_n| \leq K_n}, U_n 1_{|U_n| \leq K_n}\big)
&\le \delta_\square \big(W_n 1_{|W_n| \leq K_n}, U_n 1_{|U_n| \leq K_n}\big) +
\frac{34 K_n}{\sqrt{\log|V(G_n)|}}\\
&\le \delta_\square\big(W_n,W_n 1_{|W_n| \leq K_n}\big) + \delta_\square\big(U_n 1_{|U_n| \leq K_n},U\big)\\
&\qquad + \delta_\square(W_n,U)  +
\frac{34 K_n}{\sqrt{\log|V(G_n)|}}.
\end{align*}
Using this inequality to bound the right side of
\begin{align*}
d_\square(W_n, U) &\le d_\square(W_n, W_n 1_{|W_n| \leq K_n})
+ d_\square(U_n 1_{|U_n| \leq K_n},U)\\
&\qquad \phantom{} + d_\square(W_n 1_{|W_n| \leq K_n},U_n 1_{|U_n| \leq K_n})
\end{align*}
and bounding $\delta_\square$ by $d_\square$ yields
\begin{align*}
d_\square(W_n,U) &\le \delta_\square(W_n,U) + \frac{34 K_n}{\sqrt{\log|V(G_n)|}}\\
&\qquad \phantom{} +  2d_\square(W_n, W_n 1_{|W_n| \leq K_n})
+ 2d_\square(U_n 1_{|U_n| \leq K_n},U).
\end{align*}
To estimate $2d_\square(W_n, W_n 1_{|W_n| \leq K_n}) + 2d_\square(U_n
1_{|U_n| \leq K_n},U)$, we will bound $d_\square$ by $d_1$. We have
\begin{align*}
d_1\big(W_n, W_n 1_{|W_n| \leq K_n}\big) &\leq \norm{W_n1_{\abs{W_n} > K_n}}_1\\
& \leq \norm{W_n(\abs{W_n}/K_n)^{p-1}}_1 = \norm{W_n}_p^p / K_n^{p-1}
\leq C^p/K_n^{p-1}.
\end{align*}
Similarly,
\[
d_1\big(U_n, U_n 1_{|U_n| \leq K_n}\big) \leq \norm{U_n}_p^p /
K_n^{p-1} \leq \norm{U}_p^p / K_n^{p-1} \leq C^p/K_n^{p-1}
\]
(note that $\norm{U}_p \le C$ by Theorem~\ref{thm:G-limit}). It follows that
\[
d_\square(W_n, U) \leq \delta_\square(W_n, U) +
\frac{34 K_n}{\sqrt{\log|V(G_n)|}} + \frac{2C^p}{K_n^{p-1}} + 2d_1(U, U_n).
\]
We have $d_1(U, U_n) \to 0$ by the Lebesgue differentiation theorem, and all
the other terms tend to zero by assumption. Thus $d_\square(W_n, U) \to 0$.

Next we relax the assumption of unit node weights, and instead assume that
every vertex in $G_n$ has weight $1 + o(|V(G_n)|^{-1})$ (i.e., nearly equal
node weights).  Let $\wt W_n$ be a step function with the same values as
$W_n$, but where the step widths have all been modified to be exactly
$1/|V(G_n)|$.  We will show that $\norm{\wt W_n - W_n}_1 = o(1)$, which
suffices to reduce this case to the previous one.  Indeed, suppose the step
widths of $W_n$ are all in the interval $[1/|V(G_n)| - \alpha_n, 1/|V(G_n)| +
\alpha_n]$, where $\alpha_n|V(G_n)|^2 \to 0$. Then $\wt W_n$ and $W_n$ differ
on a set $B_n$ of measure at most $2|V(G_n)|^2\alpha_n = o(1)$, because each of
the lines separating the steps is moved by less than $|V(G_n)|\alpha_n$
(typically much less). By H\"older's inequality,
\[
\norm{\wt W_n - W_n}_1 = \int_{B_n} \abs{\wt W_n - W_n} \, d\lambda
\leq \norm{\wt W_n - W_n}_p \lambda(B_n)^{1-1/p}.
\]
We know that $\norm{W_n}_p \leq C$, and it is easy to check that $\norm{\wt
W_n}_p$ is bounded as well. Since $\lambda(B_n) \to 0$, it follows that
$\norm{\wt W_n - W_n}_1 \to 0$. This reduces the case of nearly equal node
weights to that of equal node weights.

Finally, we prove the result for a $C$-upper $L^p$ regular sequence of
weighted graphs. We may replace $C$ by a larger value if necessary and assume
that $G_n$ is $(C, \eta_n)$-upper $L^p$ regular with $\eta_n \to 0$.  Let
$\alpha_{\max}(G_n)$ denote the largest node weight in $G_n$, and recall that
$\alpha_{\max}(G_n)/\alpha_{G_n} \le \eta_n $ by
Definition~\ref{def:Lp-upper-reg}. By Remark~\ref{remark:G-upper-reg-equi},
there is some equipartition $\cP_n$ of $V(G_n)$ into $k_n$ parts, for some
slowly growing $k_n$ satisfying $k_n \to \infty$ and $k_n^2 \eta_n \to 0$, so
that $W'_n := (W_n)_{\cP_n}$ satisfies $d_\square(W'_n, W_n) \to 0$. Then
$\delta_\square(W'_n, U) \to 0$. Furthermore, $\norm{W'_n}_p \leq C$ since
$G_n$ is $(C,\eta_n)$-upper $L^p$ regular. Note that $W'_n$ is a step
function with step widths $1/k_n + o(1/k_n^2)$, since $\cP_n$ is an
equipartition into $k_n$ parts and $\alpha_{\max}(G_n) / \alpha_G =
o(1/k_n^2)$. Now we apply the case in the previous paragraph to $W'_n$ to
reorder the parts of $\cP_n$ so that $d_\square(W'_n,U) \to 0$.  If we order
the vertices of $G_n$ according to this ordering of $\cP_n$ and arbitrarily
order the vertices within each part, then $d_\square(W_n, U) \leq
d_\square(W'_n, W_n) + d_\square(W'_n, U) \to 0$, as desired.
\end{proof}

\section{$W$-random weighted graphs} \label{sec:weighted-random}

In this section and the next, we prove Theorem~\ref{thm:rand-converge}
on $W$-random graphs, thereby traversing the outer arrows of
Figure~\ref{fig:summary}. First, in this section, we address the arrow
($L^p$ graphon limit) $\to$ ($L^p$ graphon sequence) by proving Theorem~\ref{thm:rand-converge}(a), which says that
$d_1(\bH(W, n), W) \to 1$ almost surely (i.e., with probability $1$) for any $L^1$
graphon $W$.

The following theorem of Hoeffding on $U$-statistics implies that
$\norm{\bH(W,n)}_1 \to \norm{W}_1$ almost surely.

\begin{theorem}[Hoeffding \cite{H}] \label{thm:SLLN-W} Let $W \colon [0,1]^2 \to \RR$ be a
  symmetric, integrable function, and let
  $x_1, x_2, \dots$ be a sequence of i.i.d.\ random variables
  uniformly chosen from $[0,1]$. Then with probability $1$,
\[
\lim_{n \to \infty} \frac{1}{\binom{n}{2}} \sum_{1 \leq i < j \leq n} W(x_i, x_j) \to \int_{[0,1]^2} W(x,y) \,dx\,dy.
\]
\end{theorem}

\begin{proof}[Proof of Theorem~\ref{thm:rand-converge}(a)]
  All weighted random graphs $\bH(\cdot,n)$ in this proof come from
  the same random sequence $x_1, x_2, \dots$ with terms drawn
  uniformly  i.i.d.\ from $[0,1]$.

    Fix $\e > 0$. It suffices to show that
  $\limsup_{n \to \infty} d_1(\bH(W,n), W) \leq \e$ holds with
  probability $1$.

  Let $\cP$ denote the partition of $[0,1]$
  into $m$ equal intervals, where $m$ is chosen to be sufficiently
  large that $\norm{W - W_\cP}_1 \leq \e/2$. Fix this $m$ and
  $\cP$.   Since the sequence $x_1, x_2, \dots$ is equidistributed among the $m$
  intervals of $\cP$, with probability $1$ we have $d_1(\bH(W_\cP, n), W_\cP) \to 0$ as $n
  \to \infty$.

  We have $d_1(\bH(W, n), \bH(W_\cP,n)) = \norm{\bH(W - W_\cP, n)}_1$,
  which by Theorem~\ref{thm:SLLN-W} converges almost surely to
  $\norm{W - W_\cP}_1$.  It follows that, with probability $1$, the
  limit superior (as $n \to \infty$) of
  \[
  d_1(\bH(W,n), W) \leq d_1(W, W_\cP) + d_1(W_\cP, \bH(W_\cP,n)) + d_1(\bH(W_\cP,n),\bH(W,n))
  \]
  is at most $2\norm{W -
    W_\cP}_1 \leq \e$, as claimed.
\end{proof}

\section{Sparse random graphs} \label{sec:sparse-random}

In this section we prove Theorem~\ref{thm:rand-converge}(b); i.e., we prove that with
probability $1$,
$d_\square(\rho_n^{-1} \bG(n, W, \rho_n), W) \to 0$. From Theorem~\ref{thm:rand-converge}(a) we know that
$\lim_{n \to \infty} d_1(\bH(n, W), W)=0$ with probability $1$. So it remains to show that
\begin{equation}\label{eq:d(G,H)}
d_\square(\rho_n^{-1}\bG(n,W,\rho_n), \bH(n,W)) \to  0 \qquad \text{as
} n \to \infty.
\end{equation}
Here $\bG(n,W,\rho_n)$ and $\bH(n,W)$ are both generated from a common
i.i.d.\ random sequence $x_1, x_2, \ldots \in [0,1]$. We keep this assumption
throughout the section.

We will need the following variant of the Chernoff bound. The proof (a
modification of the usual proof) is included in Appendix~\ref{sec:chernoff-proof}.

\begin{lemma} \label{lem:Chernoff}
  Let $X_1, \dots, X_n$ be independent random variables, where for
  each $i$, $X_i$ is distributed as either $\mathrm{Bernoulli}(p_i)$
  or $-\mathrm{Bernoulli}(p_i)$. Let
  $X = X_1 + \cdots + X_n$ and $q = p_1 + \cdots + p_n$. Then for every $\lambda > 0$,
  \[
  \PP \paren{ \abs{X - \EE X} \geq \lambda q } \leq
  \begin{cases}
    2\exp\paren{-\frac{1}{3} \lambda^2 q} & \text{if } 0
    < \lambda \leq 1, \\
    2\exp\paren{-\frac{1}{3} \lambda q} & \text{if }
    \lambda > 1.
  \end{cases}
  \]
\end{lemma}

For a weighted graph $H$ with unit vertex weights and edge weights
$\beta_{ij} \in [-1,1]$, we use $\bG(H)$ to denote the random graph with
vertex set $V(H)$ and an edge between $i$ and $j$ with probability
$\sabs{\beta_{ij}}$, and we assign the edge weight $+1$ if $\beta_{ij} > 0$
and $-1$ if $\beta_{ij} < 0$.  (In other words, $\bG(H) = \bG(H,1)$ in the
notation of \textsection\ref{sec:defres-W-random}.)

The next two lemmas form the ($L^p$ graphon sequence) $\to$ ($L^p$
upper regular sequence) arrow in Figure~\ref{fig:summary}.

\begin{lemma} \label{eq:Gnp-cut}
  Let $\e > 0$. Let $H$ be a weighted graph on $n$ vertices with unit vertex weights, edge
weights $\beta_{ij}(H) \in [-1,1]$, and $\beta_{ii}(H) = 0$ for all $i,j \in
V(H)$. Then
  \[
  \PP\paren{d_\square(\bG(H), H) \leq \e \norm{H}_1}
  \geq
   1 - 2^{n+1}
\exp\paren{-\frac1{24} \min\{\e,\e^2\} \norm{H}_1 n^2}.
  \]
\end{lemma}

\begin{proof}
  Let $V = V(G) = V(H) = [n]$.
  For any subset $U \subseteq V$, let
  \[
  \beta_{U}(H) = \sum_{\substack{i < j \\ i,j \in U}} \beta_{ij}(H)
  \]
  be the sum of the edge weights of $H$ inside $U$. Similarly define
  $\beta_U(G)$, where $G = \bG(H)$. We also define
  \[
  \abs{\beta}_{U}(H) = \sum_{\substack{i < j \\ i,j \in U}} \abs{\beta_{ij}(H)}.
  \]
  Set
  \[
  \lambda = \frac{\e n^2\norm{H}_1}{4\abs{\beta}_U(H)}
  \geq
  \frac{\e n^2\norm{H}_1}{4\abs{\beta}_V(H)}
  = \frac{\e}{2}.
  \]
  It follows from Lemma~\ref{lem:Chernoff} that
  \begin{align*}
    \PP\paren{ \abs{\beta_U(G) - \beta_U(H)} \geq \frac{1}{4}\e n^2\norm{H}_1}
    &=
    \PP\paren{ \abs{\beta_U(G) - \beta_U(H)} \geq \lambda
      \abs{\beta}_U(H)}
    \\
    &\leq
    2\exp\paren{-\frac13 \min\{\lambda,1\}\lambda \abs{\beta}_U(H)}
    \\
    &\leq
    2\exp\paren{-\frac1{12} \min\set{\frac{\e}{2},1} \e n^2
      \norm{H}_1}
    \\
    &\leq
    2\exp\paren{-\frac1{24} \min\{\e^2,\e\} n^2 \norm{H}_1}.
  \end{align*}
By the union bound, with probability at least $1 - 2^{n+1}
\exp\paren{-\frac1{24} \min\{\e^2,\e\} n^2 \norm{H}_1}$,
\begin{equation} \label{eq:Gnp-cut-all-U}
\abs{\beta_U(G) - \beta_U(H)} \leq\frac14 \e n^2 \norm{H}_1 \quad \text{for all } U
\subseteq [n].
\end{equation}
For $S, T \subseteq V$, let
\[
\beta_{S \x T} = \sum_{s \in S, t \in T} \beta_{st}.
\]
We have
\[
\beta_{S \x T} = \beta_{S \cup T} + \beta_{S \cap T} - \beta_{S
  \setminus T} - \beta_{T \setminus S}.
\]
We deduce from \eqref{eq:Gnp-cut-all-U} that
\[
\abs{\beta_{S \x T} (G) - \beta_{S \x T}(H)} \leq \e n^2 \norm{H}_1 \quad \text{for all } S,T
\subseteq [n],
\]
which is equivalent to $d_\square(G,H) \leq \e \norm{H}_1$.
\end{proof}

The following lemma shows that $d_\square (\rho_n^{-1}\bG(H_n, \rho_n),
H_n) \to 0$ for any sequence of weighted graphs that satisfy certain mild
conditions on the edge weights. Recall the definition of the random
graph $\bG(H_n, \rho_n)$ from \S\ref{sec:defres-W-random}.

\begin{lemma} \label{eq:random-sparsify}
Let $\rho_n > 0$ with $\rho_n \to 0$ and $n \rho_n \to \infty$. For each $n$
let $H_n$ be a weighted graph with $n$ vertices all with unit vertex weights,
and containing no loops. Suppose that $\norm{H_n}_1$ is uniformly bounded and
the edge weights $\beta_{ij}(H)$ satisfy
\begin{equation} \label{eq:G-conv-cutoff-condition}
\lim_{n \to \infty} \frac{1}{n^2} \sum_{1 \leq i < j \leq n} \max\{\abs{\beta_{ij}(H_n)} -
\rho_n^{-1}, 0\} =0.
\end{equation}
Then
\[
\lim_{n \to \infty} d_\square (\rho_n^{-1}\bG(H_n, \rho_n),
H_n) = 0
\]
with probability $1$.
\end{lemma}

\begin{proof}
  Define the weighted graph $H'_n$ with edge weights
  \[
  \beta_{ij}(H'_n)
  = \sign(\beta_{ij}(H_n))\min\{\rho_n \abs{\beta_{ij}(H_n)},1\}.
  \]
So $\bG(H_n,\rho_n) = \bG(H'_n)$.  We have
\begin{equation} \label{eq:G-conv-cutoff-0}
\begin{split}
d_1(\rho_n^{-1} H'_n, H_n) &= \frac{1}{n^2} \sum_{i,j=1}^n
\sabs{\rho_n^{-1} \beta_{ij}(H_n') - \beta_{ij}(H_n)}\\
&=
\frac{1}{n^2} \sum_{i,j=1}^n \max\{\abs{\beta_{ij}(H_n)}
  - \rho_n^{-1}, 0\},
\end{split}
\end{equation}
which goes to 0 as $n \to \infty$,
by
assumption~\eqref{eq:G-conv-cutoff-condition}. It follows that
$\rho_n^{-1} \snorm{H'_n}_1 = \snorm{H_n}_1 + o(1) = O(1)$, as we
assumed that $\snorm{H_n}_1$ is uniformly bounded.
By Lemma~\ref{eq:Gnp-cut} for every $\e > 0$ we have
\begin{align*}
  \PP(d_\square(\bG(H'_n), H'_n)
  \leq \e \rho_n)
  &\geq
  1 - 2^{n+1} \exp\paren{-\frac{1}{24}\min\set{\frac{\e
        \rho_n}{\norm{H'_n}_1}, 1}\e
    \rho_n n^2}
  \\
  &\geq
  1 - 2^{n+1} \exp\paren{-\frac{1}{24}\min\set{\Omega(\e), 1}\e
    \rho_n n^2}
  \\
  &\geq 1 - 2^{-\omega(n)}
\end{align*}
as $n \to \infty$, since $n \rho_n \to \infty$. So by the Borel-Cantelli
lemma,
\[
\lim_{n \to \infty} \rho_n^{-1} d_\square(\bG(H'_n), H'_n) = 0
\]
with probability $1$. Combined with \eqref{eq:G-conv-cutoff-0} we obtain the
desired conclusion.
\end{proof}

Finally we put everything together and complete
Figure~\ref{fig:summary} with the arrow ($L^p$ graphon limit) $\to$
($L^p$ upper regular sequence).

\begin{proof}[Proof of Theorem~\ref{thm:rand-converge}(b)]
We need to show \eqref{eq:d(G,H)}. We apply
Lemma~\ref{eq:random-sparsify} with $H_n = \bH(W, n)$. By
Theorem~\ref{thm:SLLN-W}, $\norm{H_n}_1 \to \norm{W}_1$ almost
surely, so in particular $\norm{H_n}_1$ is uniformly bounded. It
remains to check~\eqref{eq:G-conv-cutoff-condition}. We have
\[
\frac{1}{n^2} \sum_{1 \leq i < j \leq n} \max\{\abs{\beta_{ij}(H_n)} -
\rho_n^{-1}, 0\}
= \frac{1}{n^2} \sum_{1 \leq i < j \leq n} \max\{\abs{W(x_i, x_j)} -
\rho_n^{-1}, 0\},
\]
which converges to 0 as $n \to \infty$ with probability $1$ by
Theorem~\ref{thm:SLLN-W}. Indeed, since $\rho_n \to 0$, for every $K >0$ the
limit superior of the above expression is bounded by
$\tfrac{1}{2}\snorm{\max\{\abs{W}- K,0\}}_1$ by Theorem~\ref{thm:SLLN-W}, and
this can be made arbitrarily small by choosing $K$ large.
\end{proof}

\begin{proof}[Proof of Corollary~\ref{cor:G-rand-converge-normalized}]
  By Theorem~\ref{thm:rand-converge}(b), $\delta_\square(\rho_n^{-1}G_n,W) \to 0$
  with probability $1$ as $n \to \infty$, and applying the theorem to $\abs{W}$ shows that
  $\rho_n^{-1} \norm{G_n}_1 \to \norm{W}_1$ with probability $1$.  It follows that
  \begin{align*}
    \delta_\square\paren{\frac{G_n}{\norm{G_n}_1},
    \frac{W}{\norm{W}_1}} &=
    \frac{\rho_n}{\norm{G_n}_1} \delta_\square \paren{\rho_n^{-1} G_n,
    \frac{\norm{G_n}_1}{\rho_n \norm{W}_1} W}\\
    &\le  \frac{\rho_n}{\norm{G_n}_1}\paren{\delta_\square \paren{\rho_n^{-1} G_n,
    W}
    + \delta_\square \paren{W,
    \frac{\norm{G_n}_1}{\rho_n \norm{W}_1} W}}\\
    &\le \frac{\rho_n}{\norm{G_n}_1}\paren{\delta_\square \paren{\rho_n^{-1} G_n,
    W}
    + \norm{W}_\square \abs{1-
    \frac{\norm{G_n}_1}{\rho_n \norm{W}_1}}}\\
    & \to 0,
  \end{align*}
  as desired.
\end{proof}

\begin{proof}[Proof of Proposition~\ref{prop:Wquasirandom}]
  By Corollary~\ref{cor:G-conv-to-Lp}, the sequence $(G_n)_{n \ge 0}$
  must be $\norm{W}_p$-upper $L^p$ regular.
  From $\delta_\square(G_n/\norm{G_n}_1, W) \to 0$ we obtain
  $\norm{W}_1 = 1$ (note that $W \ge 0$ because $G_n$ is simple),
  and by Proposition~\ref{prop:superlinear} we have $n
  \norm{G}_1 \to \infty$. It then follows from
  Corollary~\ref{cor:G-rand-converge-normalized} that
  $\delta_\square(G'_n/\norm{G'_n}_1, W) \to 0$ with probability $1$.
  By Proposition~\ref{prop:order} we can order the vertices of $G_n$
  and $G'_n$ so that $d_\square(G_n /\norm{G_n}_1, W) \to 0$ and
  $d_\square(G'_n /\norm{G'_n}_1, W) \to 0$, and thus
  \[
  d_\square\paren{\frac{G_n}{\norm{G_n}_1}, \frac{G'_n}{\norm{G'_n}_1}} \to 0,
  \]
  as desired.
\end{proof}

\section{Counting lemma for $L^p$ graphons} \label{sec:counting}

In this section we establish results relating to counting lemmas for $L^p$
graphons, as stated in \S\ref{sec:defres-count}.

We use the following generalization of H\"older's inequality from \cite{Fin}
(also see \cite[Theorem~3.1]{LZ}). This inequality played a key role in
recent work by the fourth author and Lubetzky \cite{LZ} resolving a
conjecture of Chatterjee and Varadhan~\cite{CV} on large deviations in random
graphs, which involves an application of graph limits.

\begin{theorem}[Generalized H\"older's inequality]
  \label{thm:gen-holder}
  Let $\mu_1,\ldots,\mu_n$ be probability measures on $\Omega_1,\ldots,\Omega_n$,
  respectively, and let $\mu=\prod_{i=1}^n \mu_i$ be the product measure on
  $\Omega = \prod_{i=1}^n \Omega_i$.  Let $A_1, \ldots, A_m$ be
  nonempty subsets of $[n] := \set{1, \dots, n}$ and write
  $\Omega_{A}=\prod_{\ell\in A}\Omega_\ell$ and $\mu_{A} =
  \prod_{\ell\in A}\mu_\ell$.  Let $f_i \in L^{p_i}\left(\Omega_{A_i},
    \mu_{A_i}\right)$ with $p_i\geq 1$ for each $i\in[m]$ and suppose
  in addition that $\sum_{i: \ell\in A_i} (1/p_i) \leq 1$ for each $
  \ell\in [n]$. Then
 \[
 \int \prod_{i=1}^m \left|f_i\right| \, d\mu \leq \prod_{i=1}^m
 \left(\int \left|f_i\right|^{p_i} \, d\mu_{A_i}\right)^{1/p_i}\,.
 \]
\end{theorem}

\begin{proof}[Proof of Proposition~\ref{prop:count-finiteness}]
For the first assertion, we can give an example in the form of a separable
graphon, i.e., one of the form $W(x,y) = w(x)w(y)$. Let $w \colon [0,1] \to
[0,\infty)$ be in $L^p([0,1])$ for all $p < \Delta$ but not $p = \Delta$,
e.g., $w(x) = x^{-1/\Delta}$ (and $w(0) = 0$). Then $\norm{W}_p =
\norm{w}_p^2 < \infty$ for all $p < \Delta$, but $t(F, W) = \prod_{v \in
V(G)} \norm{w}_{\deg_F(v)}^{\deg_F(v)}$, which is infinite since
$\norm{w}_\Delta = \infty$.

For the second assertion, apply Theorem~\ref{thm:gen-holder} with $n =
\abs{V(F)}$, $\Omega_i = [0,1]$, $\mu_i$ equal to Lebesgue measure, $A_1,
\dots, A_m$ the edges of $F$ (i.e., they are two-element subsets of $V(F)$),
and $p_i = \Delta$ for all $i$.
\end{proof}

\begin{lemma} \label{lem:Holder-graph}
  Let $F$ be a simple graph with maximum degree $\Delta$. Let $\Delta
  < p < \infty$ and let $q = p/(p-\Delta + 1)$. For each edge $e
  \in E(F)$, let $W_e$ be an $L^p$ graphon. Fix an edge $e_1 \in
  E(F)$. Then
  \[
  \abs{\int_{[0,1]^{\abs{V(F)}}} \prod_{ij \in E(F)} W_{ij}(x_i,x_j)
      \,dx_1\cdots d x_{\abs{V(F)}}} \leq \norm{W_{e_1}}_q \prod_{e
      \in E(F) \setminus \{e_1\}} \norm{W_e}_p.
  \]
\end{lemma}

\begin{proof}
Apply Theorem~\ref{thm:gen-holder} with $n = \abs{V(F)}$, $\Omega_i = [0,1]$,
$\mu_i$ equal to Lebesgue measure, $A_1, \dots, A_m$ the edges of $F$ (with
$A_1 = e_1$), $p_1 = q$, and $p_i = p$ for $i \geq 2$. The inequality
$\sum_{i: \ell\in A_i} (1/p_i) \leq 1$ is satisfied for each $\ell$ because
$q < p$ and $1/q + (\Delta - 1)/p = 1$ (at most one term $1/p_i$ with $\ell
\in A_i$ can equal $1/q$, the others equal $1/p$, and there are at most
$\Delta$ terms).
\end{proof}

\begin{proof}[Proof of Theorem~\ref{thm:Lp-counting}]
  Let $V(F) = \{1, 2, \dots, n\}$ and $E(F) = \{e_1, \dots,
  e_m\}$. Let $i_t, j_t$ be the endpoints of $e_t$, for $1 \leq t
  \leq m$. We may assume that $\norm{U - W}_\square \leq \e$. We have
  \begin{align*}
  t(F, U) - t(F, W)
  &= \int_{[0,1]^n}\paren{\prod_{t=1}^m U(x_{i_t},x_{j_t}) -
    \prod_{t=1}^m W(x_{i_t},x_{j_t})} \,dx_1\cdots dx_n.
  \\
  &= \sum_{t=1}^m \int_{[0,1]^n} \paren{\prod_{s<t}
    U(x_{i_s},x_{j_s})} \paren{U(x_{i_t},x_{j_t})-W(x_{i_t},x_{j_t})} \cdot\\
    & \qquad \qquad \qquad\!\! 
    \paren{\prod_{s>t}
    W(x_{i_s},x_{j_s})} \,dx_1\cdots dx_n.
\end{align*}
It suffices to show that for each $t = 1, \dots, m$,
\begin{equation}
  \label{eq:telescope-term}
\begin{split}
  \abs{\int_{[0,1]^n} \paren{\prod_{s<t}
    U(x_{i_s},x_{j_s})} \paren{U(x_{i_t},x_{j_t})-W(x_{i_t},x_{j_t})}  \paren{\prod_{s>t}
    W(x_{i_s},x_{j_s})} \,dx_1\cdots dx_n}\\
\leq 2(m-1+p-\Delta) \paren{\frac{2\e}{p-\Delta}}^{\frac{p-\Delta}{p-\Delta+m-1}}.
\end{split}
\end{equation}
Let $K > 0$, which we will choose later. Let $U = U_{\leq K} + U_{>
  K}$, where $U_{\leq K} := U 1_{\abs{U}\leq K}$ and $U_{> K} := U
1_{\abs{U}> K}$. Similarly, let $W_{\leq K} := W 1_{\abs{W}\leq K}$ and
$W_{> K} := W 1_{\abs{W}> K}$. We claim that
\begin{equation}
  \label{eq:telescope-term-cut}
\begin{split}
  \left|\int_{[0,1]^n} \paren{\prod_{s<t}
    U_{\leq K}(x_{i_s},x_{j_s})} \paren{U(x_{i_t},x_{j_t})-W(x_{i_t},x_{j_t})}\cdot\right.\\
    \left.\paren{\prod_{s>t}
    W_{\leq K}(x_{i_s},x_{j_s})} \,dx_1\cdots dx_n\right| &
\\
& \leq 4K^{m-1} \e.  
\end{split}
\end{equation}
Indeed, if we fix
the value of $x_i$ for all $i \in [n] \setminus \{i_t,j_t\}$, then  the
integral in \eqref{eq:telescope-term-cut} has the form
\begin{equation} \label{eq:telescope-term-2-var}
K^{m-1} \int_{[0,1]^2} (U(x_{i_t},x_{j_t}) - W(x_{i_t},x_{j_t})) a(x_{i_t})
b(x_{j_t}) \,dx_{i_t}\,dx_{j_t}
\end{equation}
for some functions $a(\cdot)$ and $b(\cdot)$ with $\norm{a}_\infty,
\norm{b}_\infty \leq 1$, where $a(\cdot)$ and $b(\cdot)$ depend on the values
of $x_i$ for $i \in [n] \setminus \{i_t,j_t\})$ that we fixed. Thus
\eqref{eq:telescope-term-2-var} is bounded in absolute value by $K^{m-1}
\norm{U - W}_{\infty \to 1} \leq 4K^{m-1} \e$, using
\eqref{eq:cut-to-operator-norm}. The inequality \eqref{eq:telescope-term-cut}
then follows.

Next we claim that the difference between the integral in
\eqref{eq:telescope-term} and the integral in \eqref{eq:telescope-term-cut}
is bounded in absolute value by $2(m-1)/K^{p-\Delta}$. Indeed, writing this
difference as a telescoping sum in a similar fashion to what we did at the
beginning of this proof, it suffices to show that each expression of the
following form is bounded in absolute value by $2/K^{p-\Delta}$:
\begin{equation} \label{eq:telescope-term-*}
\int_{[0,1]^n} \paren{\prod_{s<t}
    U_{*}(x_{i_s},x_{j_s})} \paren{U(x_{i_t},x_{j_t})-W(x_{i_t},x_{j_t})} \paren{\prod_{s>t}
    W_{*}(x_{i_s},x_{j_s})} \,dx_1\cdots dx_n,
\end{equation}
where we replace exactly one of the $m-1$ subscript $*$'s by `$> K$', replace
some of the other $*$'s by `$\leq K$', and then erase the remaining $*$'s.
Now we apply Lemma~\ref{lem:Holder-graph} with the special edge $e_0$
corresponding to the factor whose subscript is replaced by `$>K$'. We use
$\norm{U_{\leq K}}_p \leq \norm{U}_p \leq 1$ and $\norm{W_{\leq K}}_p \leq
\norm{W}_p \leq 1$. Using the triangle inequality we have $\norm{U - W}_p
\leq 2$. Also,
\[
\norm{U_{>K}}_q \leq \norm{U (\abs{U}/K)^{p/q - 1}}_q =
\norm{U}_p^{p/q} / K^{p/q-1}  \leq 1/K^{p-\Delta}.
\]
It then follows from Lemma~\ref{lem:Holder-graph} that an integral
of the form~\eqref{eq:telescope-term-*} is at most $2/K^{p-\Delta}$ in
absolute value.

Combining the bounds on \eqref{eq:telescope-term-cut} and
\eqref{eq:telescope-term-*}, we see that the integral in
\eqref{eq:telescope-term} is bounded in absolute value by
\[
4K^{m-1} \e + 2(m-1)/K^{p-\Delta}.
\]
We optimize this bound by choosing $K =
((p-\Delta)/(2\e))^{1/(m-1+p-\Delta)}$, which gives the bound
in \eqref{eq:telescope-term} that we claimed.
\end{proof}

Next we  give an example showing that no counting lemma can hold
when $p \leq \Delta$.

\begin{proof}[Proof of Proposition~\ref{eq:Lp-no-count}]
  By nesting of norms, we only need to consider the case $p = \Delta$.
  For for each $n\geq 1$, consider
  the separable graphon $W_n$ defined by
  \[
  W_n(x,y) := w_n(x)w_n(y),
  \]
  where $w_n(x) := 1 + u_n(x)$ with $u_n(x) := (x \ln n)^{-1/\Delta} 1_{[1/n,1]}(x)$.
  We chose $u_n$ so that it satisfies $\norm{u_n}_\Delta = 1$ and
  $\lim_{n \to \infty} \norm{u_n}_p = 0$ for $1 \leq p < \Delta$.

  We have
  \[
  \norm{W_n}_\Delta = \norm{w_n}_\Delta^2 \leq (1 +
  \norm{u_n}_\Delta)^2 = 4.
  \]
  Also, since $W_n(x,y) - 1 = u_n(x) + u_n(y) + u_n(x)u_n(y)$,
  \[
  \norm{W_n - 1}_1 \leq 2\norm{u_n}_1 + \norm{u_n}_1^2 \to 0 \quad
  \text{as } n \to \infty.
  \]
  It remains to verify that $\liminf_{n \to \infty} t(F, W_n) >
  1$. Since $W_n$ is separable,
  \[
  t(F, W_n) = \prod_{v \in V(F)} \norm{w_n}_{\deg_F(v)}^{\deg_F(v)}.
  \]
  For any integer $k$,
  \[
  \norm{w_n}_k^k = \EE [(1 + u_n)^k] = \sum_{i=0}^k \binom{n}{i} \norm{u_n}_i^i.
  \]
  Since $\norm{u_n}_\Delta = 1$ and
  $\lim_{n \to \infty} \norm{u_n}_p = 0$ for any $1 \leq p < \Delta$,
  we find that $\lim_{n\to\infty} \norm{w_n}_k^k = 1$ when $1 \le k <
  \Delta$, and $\lim_{n\to\infty} \norm{w_n}_\Delta^\Delta =
  2$. Therefore,
  \[
  \lim_{n \to \infty} t(F, W_n) = 2^{\abs{\{v \in V(G) \,:\, \deg_F(v) = 
      \Delta\}}} > 1,
  \]
  as desired.
\end{proof}

There has been some recent work by the fourth author along with Conlon and
Fox~\cite{CFZ1,CFZ2} developing counting lemmas for sparse graphs assuming
additional hypotheses. Namely one assumes that the sparse graph $G$ is a
relatively dense subgraph of another sparse graph $\Gamma$ that has certain
pseudorandomness properties. For example, to obtain a counting lemma for
$K_3$ in $G$, one assumes that $t(H, \Gamma/\norm{\Gamma}_1) = 1 + o(1)$
whenever $H$ is a subgraph of $K_{2,2,2}$ (which is the 2-blow-up of $K_3$).
More generally, an $F$-counting lemma needs $t(H, \Gamma/\norm{\Gamma}_1) = 1
+ o(1)$ whenever $H$ is a subgraph of the 2-blow-up of $F$. One might ask
whether this result can be extended to $L^p$ upper regular graphs. This is an
interesting and non-trivial problem, and we leave it open for future work.

\section*{Acknowledgments}

We thank Oliver Riordan for suggesting the topic of Appendix~\ref{sec:uniform},
Svante Janson for providing valuable feedback on our manuscript, Omer Tamuz for helpful
discussions, Remco van der Hofstad for comments about $U$-statistics that
led us to \cite{H}, Donald Cohn for advice about measure theory, and Patrick
Wolfe and Sofia Olhede for bringing \cite{WO} to our attention.

\appendix

\section{$L^p$ upper regularity implies unbounded average degree} \label{sec:superlinear}

\begin{proposition} \label{prop:superlinear}
Let $C > 0$ and $p > 1$, and let $(G_n)_{n \ge 0}$ be a $C$-upper $L^p$
regular sequence of simple graphs. Then $\abs{E(G_n)}/\abs{V(G_n)} \to
\infty$ as $n \to \infty$.
\end{proposition}

This proposition follows immediately from the following lemma.

\begin{lemma}
For every $C>0$ and $p > 1$ there exist $\eta_0 > 0$ and $c > 0$ such that
if $0 < \eta < \eta_0$ and $G$ is a $(C, \eta)$-upper $L^p$ regular simple
graph, then $\abs{E(G)}/\abs{V(G)} \geq c \eta^{-1+1/p}$.
\end{lemma}

\begin{proof}
Let $\eta_0 = \min\big((2C)^{-p/(p-1)}/2,1/3\big)$, and suppose $G$ is a $(C,
\eta)$-upper $L^p$ regular simple graph with $0 < \eta < \eta_0$.  We will
omit all floor and ceiling signs below in order to keep the notation clean.

Let $V = V(G)$, $n = \abs{V}$, and $m = \abs{E(G)}$, let $T$ be a maximal
matching (a maximal set of vertex-disjoint edges) in $G$ consisting of $t$
edges, and let $A$ be the set of vertices in $T$.  We begin by showing that
our choice of $\eta_0$ ensures $t \ge \eta_0 n$.

The proof of $t \ge \eta_0 n$ will amount to applying the definition
\eqref{eq:G-Lp} of $(C,\eta)$-upper regularity to the partition $\{A, V
\setminus A\}$.  To do so, we need both $\abs{A}$ and $\abs{V \setminus A}$
to be at least $\eta n$. If $A$ is too small, then we simply enlarge it to
have size $\eta n$; we will see below that this case never actually occurs.
We need not worry about the case when $A$ is too large, because then $t \ge
\eta_0 n$ automatically holds (since in that case $\eta_0 \le 1/3$ implies $t
= \abs{A}/2 \ge (1-\eta)n/2 \ge \eta_0 n$).

Now we can apply upper regularity.  Every edge of $G$ has a vertex in $A$
due to the maximality of $T$, and so from the partition $\{A, V\setminus
A\}$ and the $(C, \eta)$-upper $L^p$ regularity of $G$ we obtain
\begin{align*}
C^p &\ge \frac{\abs{A}^2}{\abs{V}^2}\paren{\frac{\rho_G(A,A)}{\norm{G}_1}}^p
+ \frac{\abs{A} \abs{V \setminus A}}{\abs{V}^2}\paren{\frac{\rho_G(A,V \setminus A)}{\norm{G}_1}}^p\\
&=\frac{\abs{A}^2}{\abs{V}^2}\paren{\frac{\abs{E(A)}}{\abs{A}^2 \norm{G}_1}}^p
+ \frac{\abs{A} \abs{V \setminus A}}{\abs{V}^2}\paren{\frac{\abs{E(G) \setminus E(A)}}{\abs{A} \abs{V \setminus A} \norm{G}_1}}^p\\
& \geq
\frac{\abs{A}}{\abs{V}}\paren{\frac{\abs{E(G)}}{\abs{A}\abs{V}\norm{G}_1}}^p\\
& = \frac{\abs{A}}{n} \paren{\frac{n}{2
\abs{A}}}^p,
\end{align*}
where the last inequality follows from Jensen's inequality and the convexity
of $x \mapsto x^p$.  Thus,
\[
\abs{A} \geq (2C)^{-p/(p-1)} n \ge 2 \eta_0 n
\]
and hence $t = \abs{A}/2 \ge \eta_0 n$.  (In particular, $A$ cannot have been
enlarged in the previous paragraph, because then $\abs{A} = \eta n$ would
contradict $\abs{A} \ge 2 \eta_0 n$.)

Let $\cP = \{P_1, \dots, P_{1/\eta}\}$ be a partition of $V$ into sets of
size $\eta n$ (plus at most one remainder set of size between $\eta n$ and
$2\eta n$) so that every edge of $T$ lies entirely in some part of $\cP$; in
other words, $T \subseteq \bigcup_i P_i \x P_i$. Then, by the definition of
$L^p$ upper regularity and the convexity of $x \mapsto x^p$,
\begin{align*}
2Cm/n^2 =
C \norm{G}_1 \geq \norm{G_\cP}_p
&\geq
\paren{\sum_{i=1}^{1/\eta}
\frac{\abs{P_i}^2}{\abs{V}^2}\paren{\frac{2\abs{T \cap (P_i \x
P_i)}}{\abs{P_i}^2}}^p}^{1/p}
\\
&\geq \paren{\frac{\sum_{i=1}^{1/\eta}
\abs{P_i}^2}{\abs{V}^2} \paren{\frac{2\abs{T}}{\sum_{i=1}^{1/\eta}
\abs{P_i}^2}}^p}^{1/p}
\\
&= \frac{2t}{n^{2/p} \big(\sum_{i=1}^{1/\eta}\abs{P_i}^2\big)^{(p-1)/p}}
\\
&= \Omega\paren{\frac{\eta_0 n}{n^{2/p} \big(\eta^{-1}
(n\eta)^2\big)^{(p-1)/p}}}
\\
&= \Omega\big(\eta_0\eta^{-(p-1)/p}n^{-1}\big).
\end{align*}
It follows that $m/n = \Omega_{p,C}\big(\eta^{-(p-1)/p}\big)$, as desired.
\end{proof}

\section{Proof of a Chernoff bound} \label{sec:chernoff-proof}

\begin{proof}[Proof of Lemma~\ref{lem:Chernoff}]
  Let $t = \ln(1+\lambda)$. We have
  \begin{equation} \label{eq:Chernoff-moment}
  \begin{split}
    \PP \paren{X - \EE X \geq \lambda q}
    &\leq \EE[ \exp(t (X - \EE X - \lambda q))]\\
    &= \prod_{i=1}^n \EE[\exp(t(X_i - \EE X_i - \lambda p_i))].
  \end{split}
  \end{equation}
  If $X_i$ is distributed as  $\mathrm{Bernoulli}(p_i)$, then
  \[
  \begin{split}
  \EE[\exp(t(X_i - \EE X_i - \lambda p_i))]
  &= (1 - p_i + p_i e^t) \exp(- t p_i(1 + \lambda))\\
  &\leq \exp(p_i(e^t - 1 - t(1 + \lambda))).
  \end{split}
  \]
  We have
  \[
  e^t - 1 - t(1 + \lambda) = \lambda - (1+\lambda)\ln(1+\lambda)
  \leq
    \begin{cases}
     - \frac{1}{3} \lambda^2 & \text{if } 0
    < \lambda \leq 1, \\
     - \frac{1}{3} \lambda & \text{if }
    \lambda > 1.
  \end{cases}
  \]
  On the other hand, if $X_i$ is distributed as $-\mathrm{Bernoulli}(p_i)$, then
  \[
  \begin{split}
  \EE[\exp(t(X_i - \EE X_i - \lambda p_i))]
  &= (1 - p_i + p_i e^{-t}) \exp(t p_i(1-\lambda))\\
  &\leq \exp(p_i(e^{-t} - 1 + t(1 - \lambda)))
  \end{split}
  \]
  and
  \[
  e^{-t} - 1 + t(1 - \lambda) = \frac{-\lambda}{1 + \lambda} + (1 - \lambda)\ln(1+\lambda)
  \leq
    \begin{cases}
     - \frac{1}{2} \lambda^2 & \text{if } 0
    < \lambda \leq 1, \\
     - \frac{1}{2} \lambda & \text{if }
    \lambda > 1.
  \end{cases}
  \]
  Thus in both cases,
  \[
  \EE[\exp(t(X_i - \EE X_i - \lambda p_i))] \leq
  \begin{cases}
     \exp\paren{- \frac{1}{3} \lambda^2p_i} & \text{if } 0
    < \lambda \leq 1, \\
     \exp\paren{- \frac{1}{3} \lambda p_i} & \text{if }
    \lambda > 1.
  \end{cases}
  \]
  Using these bounds in \eqref{eq:Chernoff-moment}, we find that
  \[
  \PP \paren{ X - \EE X \geq \lambda q } \leq
  \begin{cases}
    \exp\paren{-\frac{1}{3} \lambda^2 q} & \text{if } 0
    < \lambda \leq 1, \\
    \exp\paren{-\frac{1}{3} \lambda q} & \text{if }
    \lambda > 1.
  \end{cases}
  \]
  The same upper bound holds for $\PP \paren{ X - \EE X \leq - \lambda
    q }$ since it is equivalent to the previous case after negating
  all $X_i$'s. The result follows by combining the two bounds using a
  union bound.
\end{proof}

\section{Uniform upper regularity} \label{sec:uniform}

In the theory of martingale convergence, $L^p$ boundedness implies $L^p$
convergence when $p>1$, but the same is not true for $p=1$.  Instead, $L^1$
convergence is characterized by uniform integrability.  Oliver Riordan asked
whether there is a similar characterization of convergence to $L^1$
graphons.  In this appendix, we show that the answer is yes.  Although
bounding the $L^1$ norm itself is insufficient, more detailed tail bounds
suffice.  In fact, the same truncation arguments that work for $p>1$ then
extend naturally to $p=1$.

\begin{definition}
Let $K \colon (0,\infty)\to (0,\infty)$ be any function.  A graphon $W$ has
\emph{$K$-bounded tails} if for each $\varepsilon>0$,
\[
\norm{W 1_{\abs{W} \ge K(\varepsilon)}}_1 \le \varepsilon.
\]
A set $S$ of graphons is \emph{uniformly integrable} if there exists a
function $K \colon (0,\infty)\to (0,\infty)$ such that all graphons in $S$
have $K$-bounded tails.
\end{definition}

Every graphon has $K$-bounded tails for some $K$, because we have assumed as
part of our definition that all graphons are $L^1$.  For purposes of
analyzing convergence, we consider a tail bound function $K$ to be the $L^1$
equivalent of a bound on the $L^p$ norm for $p>1$.  For comparison, note
that for $K>0$,
\[
\norm{W 1_{\abs{W} \ge K}}_1 \le \norm{W \paren{\frac{\abs{W}}{K}}^{p-1}}_1 = \frac{\norm{W}_p^p}{K^{p-1}},
\]
which tends to zero as $K \to \infty$ as long as $p>1$ and $\norm{W}_p <
\infty$.

Recall that $L^1$ upper regularity is vacuous, since every graphon is $L^1$
upper regular.  To get the right notion of upper regularity, we simply
replace $L^1$ boundedness with $K$-bounded tails:

\begin{definition}
Let $K\colon (0,\infty)\to(0,\infty)$ and $\eta>0$. A graphon $W$ is
\emph{$(K,\eta)$-upper regular} if $W_{\cP}$ has $K$-bounded tails for every
partition $\cP$ of $[0,1]$ with all parts of size at least $\eta$.

A sequence $(W_n)_{n \ge 0}$ of graphons is \emph{uniformly upper regular}
if there exist $K \colon (0,\infty) \to (0,\infty)$ and
$\eta_0,\eta_1,\ldots>0$ such that $\lim_{n \to \infty} \eta_n = 0$ and
$W_n$ is $(K,\eta_n)$-upper regular.

We define $(K,\eta)$-upper regularity of a weighted graph $G$ using the
graphon $W^G/\norm{G}_1$, except that we consider only partitions $\cP$ that
correspond to partitions of $V(G)$ for which all the parts have weight at
least $\eta \alpha_G$, and we require every vertex of $G$ to have weight at
most $\eta \alpha_G$.
\end{definition}

Note that if a graph sequence has no dominant nodes and the corresponding
graphon sequence is uniformly upper regular, then so is the graph sequence.

Uniform upper regularity is the proper $L^1$ analogue of $L^p$ upper
regularity, and imposing uniform integrability avoids the otherwise
pathological behavior of $L^1$ graphons.  Our results for $L^p$ graphons
with $p>1$ then generalize straightforwardly to $L^1$. In the remainder of
this appendix, we state the results and describe the minor modifications
required for their proofs.

We will need the following two lemmas, which are standard facts about
uniform integrability and conditioning a uniformly integrable set of random
variables on different $\sigma$-algebras.

\begin{lemma} \label{lemma:unifint1}
Let $K \colon (0,\infty) \to (0,\infty)$ be any function. Then for each
$\varepsilon>0$, there exists $\delta>0$ such that for every graphon $W$
with $K$-bounded tails and every subset $I$ of $[0,1]^2$ with Lebesgue
measure $\lambda(I) \le \delta$,
\[
\int_I \abs{W} \le \varepsilon.
\]
Explicitly, $\delta$ can be chosen to be $\varepsilon/(2K(\varepsilon/2))$.
\end{lemma}

\begin{proof}
For each $I$ satisfying $\lambda(I) \le \varepsilon/(2K(\varepsilon/2))$,
\[
\norm{W 1_I}_1 \le \norm{W 1_{\abs{W} \le K(\varepsilon/2)} 1_I}_1 + \norm{W 1_{\abs{W} \ge K(\varepsilon/2)}}_1 \le K(\varepsilon/2) \lambda(I) + \varepsilon/2 \le \varepsilon.
\qedhere
\]
\end{proof}

\begin{lemma} \label{lemma:unifint2}
Let $S$ be a uniformly integrable set of graphons.  Then
\[
\set{W_\cP : \textup{$W \in S$ and $\cP$ is a partition of $[0,1]$}}
\]
is uniformly integrable.
\end{lemma}

\begin{proof}
Suppose $\norm{W}_1 \le C$ for all $W \in S$ (every uniformly integrable set
is $L^1$ bounded).  Let $\varepsilon>0$, and let $\delta$ be such that
$\norm{W 1_I}_1 \le \varepsilon$ whenever $W \in S$ and $\lambda(I) \le
\delta$, by Lemma~\ref{lemma:unifint1}.  We will show that if $K =
C/\delta$, then $\norm{W_\cP 1_{\abs{W_\cP} \ge K}}_1 \le \varepsilon$ for
all $W \in S$ and $\cP$.

Let $W$ be in $S$ and $\cP$ be a partition, and let $I$ be the set on which
$\abs{W_\cP} \ge K$.  Then
\[
K \lambda(I) \le \norm{W_{\cP}}_1 \le \norm{W}_1 \le C,
\]
and hence $\lambda(I) \le \delta$.  It follows that $\norm{W 1_{\abs{W_\cP}
\ge K}}_1 \le \varepsilon$, while $\norm{W_\cP 1_{\abs{W_\cP} \ge K}}_1 \le
\norm{W 1_{\abs{W_\cP} \ge K}}_1$ thanks to the triangle inequality
(look at each part of $\cP$). Thus,
\[
\norm{W_\cP 1_{\abs{W_\cP} \ge K}}_1 \le \varepsilon,
\]
as desired.
\end{proof}

We begin with the analogue of Proposition~\ref{prop:conv-to-Lp}.

\begin{proposition}
Let $W_0,W_1,\dots$ and $W$ be graphons such that $\delta_\square(W_n,W) \to
0$ as $n \to \infty$.  Then the sequence $(W_n)_{n \ge 0}$ is uniformly
upper regular.
\end{proposition}

It follows immediately that the same also holds for graphs, as long as they
have no dominant nodes.

\begin{proof}
Choose $\eta_n$ so that $\eta_n \to 0$ and
\[
\norm{W_n - W^{\sigma_n}}_\square \le \eta_n^3
\]
for some measure-preserving bijection $\sigma_n$ on $[0,1]$. Then
\[
\norm{\paren{W_n}_\cP - \paren{W^{\sigma_n}}_\cP}_\infty \le \eta_n
\]
whenever all the parts of $\cP$ have size at least $\eta_n$, as in
Lemma~\ref{lem:small-cut-upper-regular}.  We would like to show that picking
$K$ large enough forces $\norm{\paren{W_n}_\cP 1_{\abs{\paren{W_n}_\cP} \ge
K}}_1$ to be small.

We have
\begin{align*}
\norm{\paren{W_n}_\cP 1_{\abs{\paren{W_n}_\cP} \ge K}}_1 &\le
\norm{\paren{\paren{W^{\sigma_n}}_\cP+\eta_n} 1_{\abs{\paren{W_n}_\cP} \ge K}}_1\\
&\le \norm{\paren{\paren{W^{\sigma_n}}_\cP+\eta_n} 1_{\abs{\paren{W^{\sigma_n}}_\cP} \ge K-\eta_n}}_1.
\end{align*}
If we take $K \ge 2\eta_n$ (which is possible because $\eta_n \to 0$ as $n
\to \infty$), then we have an upper bound of
\[
2\norm{\paren{W^{\sigma_n}}_\cP 1_{\abs{\paren{W^{\sigma_n}}_\cP} \ge K-\eta_n}}_1,
\]
which tends uniformly to zero as $K \to \infty$ by
Lemma~\ref{lemma:unifint2}.
\end{proof}

The converse is also true: every uniformly upper regular sequence has a
convergent subsequence (Theorem~\ref{thm:unifintlimit}). This is the
analogue of Theorem~\ref{thm:W-upper-reg-limit}, but we will have to develop
machinery for the $L^1$ case before we can prove it.

\begin{theorem}[Weak regularity lemma] \label{thm:weakregL1}
Fix $K \colon (0,\infty) \to (0,\infty)$. For each $\varepsilon>0$, there
exists an $N$ such that for every natural number $k \ge N$, every graphon
$W$ with $K$-bounded tails, and every equipartition $\cP$ of $[0,1]$, there
exists an equipartition $\cQ$ refining $\cP$ into $k \abs{\cP}$ parts such
that
\[
\norm{W - W_\cQ}_\square \le \varepsilon.
\]
\end{theorem}

\begin{proof}
We start by applying the $L^2$ weak regularity lemma
(Lemma~\ref{lem:W-L^2-reg-equi}) to the truncation $W 1_{\abs{W} \le
K(\varepsilon/4)}$, which has $L^2$ norm at most $K(\varepsilon/4)$.  It
follows that the theorem statement holds with the conclusion replaced by
\[
\norm{W 1_{\abs{W} \le K(\varepsilon/4)} - \paren{W 1_{\abs{W} \le K(\varepsilon/4)}}_\cQ}_\square \le \varepsilon/4.
\]
Thus, for $U := W 1_{\abs{W} \le K(\varepsilon/4)}$ we can find a $\cQ$ such
that
\[
\norm{U-U_\cQ}_\square \le \varepsilon/4.
\]
Then
\[
\norm{W-U_\cQ}_\square \le \norm{W 1_{\abs{W} \ge K(\varepsilon/4)}}_1 + \norm{U-U_\cQ}_\square \le \varepsilon/2,
\]
from which it follows that $\norm{W-W_\cQ}_\square \le \varepsilon$ (see the
end of Remark~\ref{remark:graphon-equi} for this standard inequality).
Thus, the same partitions that give an $\varepsilon/4$-approximation of $U$
give an $\varepsilon$-approximation of $W$.
\end{proof}

The compactness of the $L^p$ ball (Theorem~\ref{thm:compact}) requires
uniform integrability when $p=1$:

\begin{theorem} \label{thm:unifintcompact}
Let $(W_n)_{n \ge 0}$ be uniformly integrable sequence of graphons. Then
there exists a graphon $W$ such that
\[
\liminf_{n \to \infty} \delta_\square(W_n, W) = 0.
\]
\end{theorem}

\begin{proof}
The proof is almost the same as that of Theorem~\ref{thm:compact} with
$p=1$, but it uses the martingale convergence theorem for uniformly
integrable martingales \cite[Theorem~5.5.6]{Dur}, rather than $L^p$
martingales, and it uses Theorem~\ref{thm:weakregL1} for weak regularity.
The only substantive difference is in verifying that the martingale
$U_1,U_2,\dots$ is uniformly integrable (using the notation from the proof).
To do so, we start by observing that the graphons $W_{n,k}$ are uniformly
integrable by Lemma~\ref{lemma:unifint2}. Now uniform integrability for
$U_k$ follows straightforwardly, since $W_{n,k}$ converges pointwise to
$U_k$ as $n \to \infty$ and has only $\abs{\cP_k}$ parts.
\end{proof}

\begin{corollary}
Every set of graphons that is uniformly integrable and closed under the cut
metric is compact under that metric.
\end{corollary}

We will also need analogues of the results of
Section~\ref{sec:reg-lem-up-reg} for uniform upper regularity. The analogues
of Lemmas~\ref{lem:small-cut} and~\ref{lem:W-stab} are straightforward (they
use Lemma~\ref{lemma:unifint1} to replace H\"older's inequality):

\begin{lemma}
Let $K\colon (0,\infty)\to (0,\infty)$ and  $\e>0$.  Then there exists a
constant $\eta_0=\eta_0(K,\e)$ such that the following holds for all
$\eta\in (0,\eta_0)$: if $W \colon [0,1]^2 \to \RR$ is a $(K,\eta)$-upper
regular graphon and $S, T \subseteq [0,1]$ are measurable subsets with
$\lambda(S) \leq\eta_0$, then
\[
\abs{\ang{W, 1_{S \x T}}} \leq \e.
\]
\end{lemma}

\begin{lemma}
\label{lem:W-stab-L1} Let $K\colon (0,\infty)\to (0,\infty)$ and  $\e>0$.
Then there exists a constant $\eta_0=\eta_0(K,\e)$ such that the following
holds for all $\eta\in (0,\eta_0)$ and every $(K,\eta)$-upper regular
graphon $W$: if $S, S', T, T' \subseteq [0,1]$ are measurable sets
satisfying $\lambda(S \symmdiff S'), \lambda(T \symmdiff T') \leq \eta_0$,
then
\[
\abs{\ang{W, 1_{S \x T} - 1_{S' \x T'}}} \leq \e.
\]
\end{lemma}

Using these two lemmas, one can then prove the analogue of
Theorem~\ref{thm:W-upper-reg}.  Indeed, \eqref{eq:W-alter1} and
\eqref{eq:W-alter2} (with $C=1$) follow from Lemmas~\ref{lem:W-stab-L1} and
\ref{lemma:unifint1}, leading again to a bound of the form
\eqref{eq:W-cut-dev'}.  Once \eqref{eq:W-cut-dev'} is established, the proof
then just proceeds as in the truncation argument in Case~II of the proof of
Theorem~\ref{thm:W-upper-reg} by setting $U=W_{\cP_n} 1_{\abs{W_{\cP_n}}\leq
K(\tilde\e)}$ for some suitable $\tilde\e$.  This leads to the following
theorem:

\begin{theorem} [Weak regularity lemma for $(K,\eta)$-upper regular graphons]
Let $K\colon (0,\infty)\to(0,\infty)$ and $0<\e<1$. Then there exist
constants $N=N(K,\e)$ and $\eta_0=\eta_0(K,\e)$ such that the following
holds for all $\eta\leq \eta_0$: for every $(K,\eta)$-upper regular graphon
$W$, there exists a partition $\cP$ of $[0,1]$ into at most $4^N$ measurable
parts, each having measure at least $\eta$, so that
\[
\norm{W - W_\cP}_\square \leq  \e.
\]
\end{theorem}

Following the strategy leading to Remark~\ref{remark:graphon-equi} for
graphons, and that leading to Theorem~\ref{thm:G-upper-reg-lemma} and
Remark~\ref{remark:G-upper-reg-equi} for graphs, one then gets the following
version involving equipartitions and holding also for graphs.

\begin{theorem} \label{thm:weakregunif}
Let $K\colon (0,\infty)\to (0,\infty)$ and $0<\e<1$. Then there exist
constants $N=N(K,\e)$ and $\eta_0=\eta_0(K,\e)$ such that the following
holds for all $\eta \le \eta_0$: for every $(K,\eta)$-upper regular graphon
$W$ and each natural number $k \ge N$, there exists a equipartition $\cP$ of
$[0,1]$ into $k$ parts so that
\[
\norm{W - W_\cP}_\square \leq \e.
\]
The same holds for a weighted graph $G$ with $W = W^G/\norm{G}_1$, in which
case we can use an equipartition of the vertex set, as in
Remark~\ref{remark:G-upper-reg-equi}.
\end{theorem}

Theorem~\ref{thm:weakregunif} now allows us to prove the analogue of
Theorem~\ref{thm:G-limit} and~\ref{thm:W-upper-reg-limit}.

\begin{theorem} \label{thm:unifintlimit}
Every uniformly upper regular sequence of graphons or weighted graphs has a
subsequence that converges to an $L^1$ graphon under the normalized cut
metric.
\end{theorem}

The proof is almost identical to that of Theorems~\ref{thm:G-limit}
and~\ref{thm:W-upper-reg-limit}: we use the transference theorem
(Theorem~\ref{thm:weakregunif}) to reduce to the compactness theorem
(Theorem~\ref{thm:unifintcompact}).

Finally, we conclude by noting that the proofs of
Propositions~\ref{prop:order}, \ref{prop:superlinear}, and
\ref{prop:Wquasirandom} carry over to uniform upper regularity:

\begin{proposition} \label{prop:orderunif}
Let $(G_n)_{n \ge 0}$ be a uniformly upper regular sequence of weighted
graphs with $\delta_\square(G_n/\norm{G_n}_1, W) \to 0$ for some graphon
$W$. Then the vertices of the graphs $G_n$ may be ordered in such a way that
$\norm{W^{G_n}/\norm{G_n}_1 - W}_\square \to 0$.
\end{proposition}

\begin{proposition} \label{prop:superlinearunif}
Let $(G_n)_{n \ge 0}$ be a uniformly upper regular sequence of simple
graphs. Then $\abs{E(G_n)}/\abs{V(G_n)} \to \infty$ as $n \to \infty$.
\end{proposition}

\begin{proposition} \label{prop:Wquasirandomunif}
  Let $W$ be any graphon, and let $(G_n)_{n \ge 0}$ be a
  sequence of simple graphs such that $\norm{G_n}_1 \to 0$ and
  $\delta_\square(G_n/\norm{G_n}_1, W) \to 0$.
  Let $G'_n = \bG(\abs{V(G_n)}, W, \norm{G_n}_1)$. Then with probability $1$,
  one can order the vertices of $G_n$ and $G'_n$ so that
  \[
  d_\square\paren{\frac{G_n}{\norm{G_n}_1}, \frac{G'_n}{\norm{G'_n}_1}} \to 0.
  \]
\end{proposition}

The only substantive modification required for the proofs is that the $L^p$
upper regularity and convexity arguments in the proof of
Proposition~\ref{prop:superlinear} must be replaced with applications of
Lemma~\ref{lemma:unifint1}.

\end{document}